\documentclass[11pt]{article}
\usepackage{amsmath,amsthm,amssymb,dsfont,mathrsfs}
\usepackage[usenames,dvipsnames]{xcolor}
\usepackage{upgreek}
\usepackage{enumerate}
\usepackage{graphicx}
\usepackage{cite}
\usepackage{url}
\usepackage{oands}
\usepackage{tikz}
\usepackage{changepage}
\usepackage{bbm}
\usepackage{mathtools}
\usepackage[margin=1in]{geometry}
\usepackage[pagewise,mathlines]{lineno}
\usepackage{appendix}
\usepackage{stmaryrd} 
\usepackage{multicol}
\usepackage{microtype}
\usepackage[colorinlistoftodos]{todonotes}
\usepackage{subfigure}

\definecolor{darkblue}{rgb}{0.0, 0.2, 0.6}
\usepackage[pdftitle={},
  pdfauthor={},
colorlinks=true,linkcolor=RoyalBlue,urlcolor=MidnightBlue,citecolor=darkblue,bookmarks=true,bookmarksopen=true,bookmarksopenlevel=2,unicode=true,linktocpage]{hyperref}
\usepackage[capitalise,noabbrev]{cleveref}

\newcommand{\ssb}{\begin{adjustwidth}{2.5em}{0pt}}
\newcommand{\sse}{\end{adjustwidth}}

\usepackage[utf8]{inputenc} % Required for inputting international characters
\usepackage[T1]{fontenc} % Output font encoding for international characters

\usepackage
[final]
{changes}
\makeatletter
\@namedef{Changes@AuthorColor}{magenta}
\colorlet{Changes@Color}{magenta}
\makeatother
\usepackage{enumitem}
\usepackage{float}
\usepackage[labelfont={bf,sf}, textfont=sf]{caption}
\usepackage{chngcntr}
\newcommand\N{\mathbbm{N}}

\newcommand\R{\mathbbm{R}}

\newcommand\Ub{\mathbbm{U}}
\newcommand\Sb{\mathbb{S}}

\newcommand\Pb{\mathbbm{P}}
\newcommand\Eb{\mathbbm{E}}
\newcommand{\n}{\mathfrak{n}}
\newcommand{\efrak}{\mathfrak{e}}
\newcommand{\Mcal}{\mathcal{M}}
\newcommand{\e}{\mathfrak{e}}
\newcommand{\ealpha}{\mathfrak{e}^{\alpha}}
\newcommand{\Pcal}{\mathcal{P}}
\newcommand{\Ecal}{\mathcal{E}}
\newcommand\Ebf{\mathbf{E}}
\newcommand\Pbf{\mathbf{P}}
\newcommand\Phat{\widehat{\mathcal{P}}}
\newcommand\Ehat{\widehat{\mathcal{E}}}
\newcommand{\nalpha}{\mathfrak{n}^{\alpha}}
\newcommand{\gammaalpha}{\gamma^{\alpha}}

\newcommand{\Xcal}{\mathcal{X}}
\newcommand{\Xbf}{\mathbf{X}}
\newcommand{\Lcal}{\mathcal{L}}
\newcommand{\Xhat}{\widehat{\mathcal{X}}}
\newcommand{\Xbfhat}{\widehat{\mathbf{X}}}
\newcommand{\xihat}{\widehat{\xi}}
\newcommand{\Thetahat}{\widehat{\Theta}}
\newcommand{\Thetaomega}{\Theta^{[\omega]}}
\newcommand{\xiomega}{\xi^{[\omega]}}

\newcommand{\Ltilde}{\widetilde{L}}
\newcommand{\Lhat}{\widehat{L}}

\newcommand\psihat{\widehat{\psi}}

\newcommand{\Fscr}{\mathscr{F}}
\newcommand{\Fcal}{\mathcal{F}}
\newcommand{\Gcal}{\mathcal{G}}
\newcommand{\Gscr}{\mathscr{G}}

\newcommand\Jcal{\mathcal{J}}

\newcommand\Hhat{\widehat{H}}

\newcommand{\xbf}{\mathbf{x}}
\newcommand{\Hcal}{\mathcal{H}}

\newcommand{\Ccal}{\mathcal{C}}
\newcommand{\Xscr}{\mathscr{X}}

\newcommand{\Pbb}[3]{\Pb_{#1}^{#2\rightarrow #3}}
\newcommand{\Ebb}[3]{\Eb_{#1}^{#2\rightarrow #3}}

\theoremstyle{plain}
\newtheorem{Thm}{Theorem}[section]
\newtheorem*{Thm*}{Theorem}
\newtheorem{Prop}[Thm]{Proposition}
\newtheorem*{Prop*}{Proposition}
\renewenvironment{proof}{{\bfseries Proof.}}{\qed}

\newtheorem*{Def*}{Definition}

\newtheorem*{Cor*}{Corollary}

\theoremstyle{definition}

\newtheorem*{Ex*}{Example}
\newtheorem{Rk}[Thm]{Remark}

\begin{document}

\vglue30pt

%title 
\begin{center}
    \Large\bf Spatial growth-fragmentations and excursions from
hyperplanes
\end{center}

\bigskip

\centerline{by}

\medskip

\centerline{William Da Silva\footnote{\scriptsize University of Vienna, Austria, {\tt william.da.silva@univie.ac.at}} and Juan Carlos Pardo\footnote{\scriptsize Centro de Investigación en Matemáticas A.C. Calle Jalisco s/n. 36240 Guanajuato, México, {\tt jcpardo@cimat.mx}}}

\bigskip

\bigskip

{\leftskip=2truecm \rightskip=2truecm \baselineskip=15pt \small

%\noindent{\slshape\bfseries Declarations.}
%Not applicable.

\noindent{\slshape\bfseries Summary.}  In this paper, we are interested in the  self-similar growth-fragmentation process that shows up when slicing half-space excursions of a $d$-dimensional Brownian motion from hyperplanes.  Such a family of processes turns out to be a  spatial self-similar growth-fragmentation processes driven by an isotropic self-similar Markov process. The former can be seen as multitype growth-fragmentation processes, in the sense of \cite{DP23}, where the set of types is $\Sb^{d-2}$, the $(d-1)$--dimensional unit sphere. In order to characterise such family of processes, we study their spinal description similarly as in the monotype \cite{Ber-GF}  and multitype \cite{DP23}  settings. Finally, we extend our study to the case when the $d$-dimensional Brownian motion is replaced by an isotropic  Markov process whose first $(d-1)$ coordinates are driven by an isotropic stable L\'evy process and the remaining coordinate is  an independent  standard real-valued Brownian motion.  
%Write summary/abstract

\medskip


%\medskip
 
%\noindent{\slshape\bfseries 2010 Mathematics Subject
%Classification.} 
% Write classification

} %%%%%% End of narrower

\bigskip
\bigskip

%-------------------------------------------------------------------------%
%                    INTRODUCTION
%-------------------------------------------------------------------------%
\section{Introduction and main results}

Let  us  consider  a $d$-dimensional Brownian motion for $d\ge  2$, here denoted by   $B^d=(B^d_t, t\ge 0)$. 
In this paper,  we are interested in a typical Brownian excursion above the  hyperplane $\Hcal=\{(x_1, \ldots, x_d)\in \R^{d}: x_d=0\}$  starting from $\overline{0}$ and ending at $z\in \Hcal$. Such type of excursions were first studied by Burdzy in \cite{Bur} and can be easily described by considering in the last coordinate 
 a positive Brownian excursion and in the first  $d-1$ coordinates  a  $(d-1)$-dimensional Brownian motion stopped at the lifetime of the positive
 Brownian excursion. Brownian excursions away from  the  hyperplane $\Hcal$ form a Poisson point process with characteristic measure $\n$ defined on the space of continuous functions starting from  $\overline{0}$ and ending in a given point in $\Hcal$. Moreover the latter can be described in terms of  the product of the one-dimensional  It\^o excursion measure and the probability of   a $(d-1)$-dimensional Brownian motion stopped at the lifetime of the one-dimensional excursion, see Section \ref{sec: excursion}  for further details.

 Motivated by the recent work of A\"id\'ekon and Da Silva \cite{AD} in the two dimensional case, we study the family  of processes that arises when slicing upper half-space excursions  from hyperplanes of the form $\Hcal_a=\{(x_1, \ldots, x_d)\in \R^{d}: x_d=a\}$, for $a>0$. That is, if the excursion from $\Hcal$ hits $\Hcal_a$, it will make a countable number of excursions above it, here denoted by $(e^{a,+}_i, i\ge 1)$. For every excursion $e^{a,+}_i$,  we let $\Delta e^{a,+}_i$ be the difference between the endpoint and the starting point of  $e^{a,+}_i$. Since both points  are in the same hyperplane, we observe that $\Delta e^{a,+}_i$ is a vector in $\mathbb{R}^{d-1}$ and therefore the family 
 $(\Delta e^{a,+}_i; i\ge 1)$ is a collection of vectors in $\mathbb{R}^{d-1}$ that we suppose to be ranked in decreasing order of the norm. The main result of this paper describes the law of the process $(\Delta e^{a,+}_i; i\ge 1)$ indexed by $a$ in terms of, what we have called,   {\it spatial self-similar growth-fragmentations} (in $\R^{d-1}$). Spatial self-similar growth-fragmentations $\in \R^{d}$ are extensions of multitype self-similar growth-fragmentations, recently introduced by the authors in \cite{DP23}, with the difference that the set of types is given by $\Sb^{d-1}$, the $d$-dimensional unit sphere. 

Let $X=(X_t, t\ge 0)$ be an isotropic self-similar  Markov process in $\mathbb{R}^d$. The construction of the associated spatial self-similar growth-fragmentation is very similar to   \cite{Ber-GF} and much simpler than the multitype case  \cite{DP23}.  Roughly speaking, a spatial self-similar growth-fragmentation describes a cloud of elements in $\mathbb{R}^d$, that we may refer to as atoms,  which may grow and dislocate in a binary way. Initially, the cloud of atoms starts from one particle (the common ancestor of all future particles)  whose ($d$--dimensional) \emph{size} is given by the process $X$, and that will split in a binary way whenever $X$ has a jump.   More precisely, at each jump of size $y=X(t)-X(t^-)$, we add to the cloud at time $t$ a new particle with initial size $-y$. These binary divisions are conservative, in the sense that the size of the child and the size of the parent just after division exactly sum up to the size of the parent before division. The newborn particles evolve independently of the parent, and independently of each another, according to a copy of $X$. The next generations are constructed in the same way by repeating the same division process for each new individual in the cloud. We are then interested in the collection $\Xbf(a)$ of sizes of cells alive at time $a\ge 0$, ranked in a decreasing order of their norm.  

Self-similar growth-fragmentation processes first appeared in \cite{Ber-GF}  and recently extended to the mutitype setting, first in \cite{AD} and subsequently in \cite{DS} and \cite{DP23}. Importantly, the self-similar growth-fragmentation process we just defined is not included in the aforementioned works. We refer to Section  \ref{construction}  for a formal construction of such  object and for the precise statement of its branching structure.

In order to state the first main result of this paper, let us first introduce some notation. Let $\n_+$ be the excursion measure of the excursion above the hyperplane $\Hcal$, that is the restriction of $\n $ to $\Hcal^+=\{(x_1, \ldots, x_d)\in \R^{d}: x_d>0\}$, 
 and  introduce the following family of measures  $\gamma_{\xbf}$, $\xbf\in\R^{d-1}$, which are associated with the Brownian excursions from the hyperplane $\Hcal$ \emph{conditioned} on ending at $(v,0)$, by disintegrating $\n_+$ over its endpoint. Whenever $r\ge 0$ and $\xbf\in \R^{d-1}$, we write $\Pi_r$ for the law of a Bessel bridge from $0$ to $0$ over $[0,r]$, and $\Pbb{r}{0}{\xbf}$ for the law of a $(d-1)$--dimensional Brownian bridge from $0$ to $\xbf$ with duration $r$. See \cite{AD} for the  case $d=2$. For all $\xbf\in \mathbb{R}^{d-1}\setminus\{0\}$, we let 
\[
\gamma_{\xbf} := \int_0^{\infty} \mathrm{d}r \frac{\mathrm{e}^{-\frac{1}{2r}}}{2^{\frac{d}{2}}\Gamma\left(\frac{d}{2}\right) r^{\frac{d}{2}+1}} \Pbb{r|\xbf|^2}{0}{\xbf} \otimes \Pi_{r|\xbf|^2}.
\]

\begin{Thm}\label{thm:exc} Under $\gamma_{\xbf}$, the process $(\Delta e^{a,+}_i; i\ge 1)$ is a spatial self-similar growth-fragmentation in $\R^{d-1}$ whose distribution is the same  as the  spatial self-similar growth-fragmentation  described by a $(d-1)$--dimensional isotropic Cauchy process.
\end{Thm}

The previous result can be extended to the case when the $d$-dimensional Brownian motion is replaced by a process   that we denote by $Z^d=(X^{d-1}, Z)$, where $X^{(d-1)}$ is a $(d-1)$-dimensional isotropic $\alpha$-stable L\'evy process, with  $\alpha\in (0,2)$, and $Z$ is an independent real-valued Brownian motion. The excursions of the process $Z^d$ away from  the  hyperplane $\Hcal$ also form a Poisson point process with characteristic measure $\n^\alpha$ defined on the space of c\'adl\'ag functions starting from  $\overline{0}$ and ending in a given point in $\Hcal$. Similarly as in the Brownian case,  the latter can be described in terms of  the product of the one-dimensional  It\^o excursion measure and the law of   a $(d-1)$-dimensional isotropic stable process stopped at the lifetime of the one-dimensional Brownian excursion, see Section \ref{sec: excursion}  for further details.

Similarly as in the Brownian case, we let $\n^\alpha_+$ be the excursion measure of the excursions of the process $Z^d$ above the hyperplane $\Hcal$, that is the restriction of $\n^\alpha $ to $\Hcal^+$, 
 and  introduce the following family of measures  $\gamma^\alpha_{\xbf}$, $\xbf\in\R^{d-1}$ which are associated with the  excursions of $Z^{d}$ from the hyperplane $\Hcal$ \emph{conditioned} on ending at $(v,0)$, by disintegrating $\n^\alpha_+$ over its endpoint. For $\xbf\in\R^{N-1}$ and $r>0$, let $\Pbb{r}{\alpha, 0}{\xbf}$ denote the law of an $\alpha$--stable bridge from $0$ to $\xbf$ over $[0,r]$. See \cite{DS} for the case $d=2$. In addition, we write $(p^{\alpha}_r, r\ge 0)$ for the transition densities of $X^{d-1}$. For all $\xbf\in \mathbb{R}^{d-1}\setminus\{0\}$, we let 
\[
\gammaalpha_{\xbf} = \int_0^{\infty} \mathrm{d}r \frac{p_1^{\alpha}(r^{-1/\alpha}v\cdot\mathbf{1})}{2\sqrt{2\pi}r^{1+\frac{\omega_d}{\alpha}}} \Pbb{r|\xbf|^2}{\alpha, 0}{\xbf} \otimes \Pi_{r|\xbf|^2},
\]
where $\mathbf{1}$ denotes the ``north pole'' in $\Sb^{d-2}$ and $\omega_d:= d-1+\frac{\alpha}{2}$. Recall that $(\Delta e^{a,+}_i; i\ge 1)$ denotes the collection of vectors in $\mathbb{R}^{d-1}$ obtained by slicing the excursion at height $a$.

\begin{Thm}\label{thm:excal} Under $\gamma^\alpha_{\xbf}$, the process $(\Delta e^{a,+}_i; i\ge 1)$ is a spatial self-similar growth-fragmentation whose distribution is the same  as the  spatial self-similar growth-fragmentation  described by an isotropic $(d-1)$--dimensional $\frac{\alpha}{2}$--stable process.
\end{Thm}

\medskip
{\bf Related works.} 
Growth-fragmentation models have deep connections to random geometry.  The first connections were made by Bertoin, Curien and Kortchemski \cite{BCK} and Bertoin, Budd, Curien and Kortchemski \cite{BBCK}, who pointed out a remarkable class of (positive) self-similar growth-fragmentations obtained from the scaling limit of perimeter processes (see \cite{Budd}) in peeling explorations of Boltzmann planar maps. These growth-fragmentation processes are closely related to stable L\'evy processes with stability parameter  $\theta \in (\frac12, \frac32]$. Later, Miller, Sheffield and Werner constructed the same growth-fragmentation processes in the continuum \cite{MSW} for $1<\theta\le \frac32$ from a CLE exploration on a quantum disc. Moreover, the boundary case $\theta=\frac32$ shows up in the metric construction of Le Gall and Riera \cite{LG-Rie} when slicing a Brownian disc at heights. As already mentioned above, the critical Cauchy case $\theta=1$ appears when slicing a Brownian half-plane excursion at heights \cite{AD} and the case $\frac12<\theta<1$  by considering  half-plane excursions of $Z^2=(X,Z)$, where $X$ is an $\alpha$-stable L\'evy process, for $\alpha\in (1,2)$ and $Z$ and independent Brownian motion, see \cite{DS}. 

We stress that the masses in \cite{AD} and \cite{DS} are signed, with the sign depending on the time orientation of the excursions. In particular, the growth-fragmentation of \cite{DS} only yields the critical case $\theta=1$ in \cite{BBCK} after removing the negative mass in the system. Geometrically, this negative mass encodes the loops of a loop $O(n)$--model, whose gasket is described by the Boltzmann planar maps studied in \cite{BCK,BBCK}. This prompted \cite{DS} to provide a framework for self-similar \emph{signed} growth-fragmentations.

Finally, we mention that the notion of self-similar signed growth-fragmentations was recently naturally extended to a finite number of types in \cite{DP23} using deeply the interplay between self-similar Markov processes and Markov additive processes, see \cite{KP}. The present paper uses a similar viewpoint, albeit more intricate due to the presence of uncountably many types.

\bigskip
The paper is organised as follows. In Section \ref{sec: ssmp}, we provide some background on isotropic self-similar Markov processes and their Lamperti-Kiu representation in terms of Markov additive processes. In Section \ref{sec: spatial GF}, spatial growth-fragmentation processes are formally constructed. Moreover, {\it isotropic cumulant functions} are introduced whose roots will lead to martingales for spatial growth-fragmentations.  The spinal decomposition will be described in Section \ref{sec: spine isotropic}. Finally, the proofs of Theorems \ref{thm:exc} and \ref{thm:excal} are given in Section \ref{sec: excursion}.  

\bigskip
\noindent \textbf{Acknowledgments.}
We thank Andreas Kyprianou for some interesting discussions. W.D.S. acknowledges the support of the Austrian Science Fund (FWF) grant P33083-N on ``Scaling limits in random conformal geometry''.  J.C.P.  acknowledges the support of CONACyT grant A1-S-33854.

%-------------------------------------------------------------------------%
%                       ToC
%-------------------------------------------------------------------------%
%\tableofcontents

%-------------------------------------------------------------------------%
%               SELF-SIMILAR MARKOV PROCESSES IN R^d
%-------------------------------------------------------------------------%
\section{Isotropic self-similar Markov processes} \label{sec: ssmp}
We start with some background on self-similar Markov processes in $\R^d$. We lay emphasis on their connection to Markov additive processes, through the Lamperti-Kiu representation. We refer to \cite{KP} for a detailed treatment of these questions.

%-------------------------------------------------------------------------%
%                            MAPs
%-------------------------------------------------------------------------%
\bigskip
\noindent \textbf{Markov additive processes.}
Let $E$ be  a locally compact, complete and separable metric space, endowed with a cemetery state $\dagger$. We also let $(\xi(t),\Theta(t), t\ge 0)$ be a regular Feller process in $\R\times E$ with probabilities ${\tt P}_{x,\theta}$, $x\in\R$, $\theta\in E$, on $(\Omega,\Fcal,\Pb)$, and denote by $(\Gcal_t)_{t\ge 0}$ the natural standard filtration associated with $(\xi,\Theta)$. We say that $(\xi,\Theta)$ is a \emph{Markov additive process} (MAP for short) if for every bounded measurable $f: \R\times E \rightarrow \R$, $s,t\ge 0$ and $(x,\theta)\in\R\times E$,
\[
\mathtt{E}_{x,\theta}\Big[f(\xi(t+s)-\xi(t),\Theta(t+s))\mathds{1}_{\{t+s<\varsigma\}} \Big| \Gcal_t\Big] = \mathds{1}_{\{t<\varsigma\}}\mathtt{E}_{0,\Theta(t)}\Big[f(\xi(s),\Theta(s))\mathds{1}_{\{s<\varsigma\}}\Big],
\]
where $\varsigma := \inf\{t>0, \, \Theta(t)= \dagger\}$. Observe that the process $\Theta$ is itself a regular Feller process in $E$. 

We call $\xi$ the \emph{ordinate} and $\Theta$ the \emph{modulator} of the MAP. The notation 
\[
\mathtt{P}_{\theta}:=\mathbb{P}(\;\cdot\; |\,\xi(0)=0 \,\text{and}\, \Theta(0)=\theta)\quad\textrm{ for }\quad \theta\in E,
\] will be in force throughout the paper. Whilst  MAPs have found a prominent role in e.g. classical applied probability models for queues and dams when $\Theta$ is a Markov chain with a finite state space (see for instance \cite{Asm} and \cite{Iva}), the case that $\Theta$ is a general Markov process has received somewhat less attention. However, this case has been treated in the literature before, see for instance \cite{Cin} and references therein.

MAPs are a natural extension of a L\'evy process in the sense that $\Theta$ is an arbitrary well-behaved Markov process and $((\xi(t), \Theta(t))_{t\ge 0}, \mathtt{P}_{x,\theta})$ is equal in law to  
$((\xi(t)+x, \Theta(t))_{t\ge 0}, \mathtt{P}_{\theta})$. Moreover when $\Theta$ is a Markov chain with a finite state space a more natural description can be given for the ordinate process $\xi$. Indeed it can be thought as the concatenation of L\'evy processes  which depend on the current type in $E$ given by $\Theta$. Here we are interested in the specific case when $E=\Sb^{d-1}$ which describes the angles of a process in $\R^d$.

%-------------------------------------------------------------------------%
%                            SSMP in R^d
%-------------------------------------------------------------------------%
\bigskip
\noindent \textbf{Self-similar Markov processes in $\R^d$ and isotropy.}
Let  $E=\Sb^{d-1}$ and $\alpha\in \R$.  
Let $X$ be a Markov process in $\R^d$, which under $\Pb_{\xbf}$, $\xbf\in\R^d\setminus\{0\}$, starts from $\xbf$.
We say that $X$ is a \emph{self-similar} process if for all $c>0$ and all $\xbf\in\R^d$, 
\[
\Big( (cX(c^{-\alpha}s), s\ge 0), \Pb_{\xbf}\Big) \overset{d}{=} \Big( (X(s), s\ge 0), \Pb_{c\xbf}\Big).
\]
The Lamperti representation of $\alpha$-self-similar $\R^d$--valued Markov processes is the content of the following proposition which is attributed to \cite{Kiu} with additional clarification from  \cite{ACGZ}, building on the original work of \cite{Lam}.
\begin{Prop} \label{prop: Lamperti R^d}
Let $X$ be a self-similar $\R^d$--valued Markov process with index $\alpha$. Then there exists a Markov additive process $(\xi,\Theta)$ in $\R\times \Sb^{d-1}$ such that
\begin{equation} \label{eq: Lamperti R^d}
X(t) = \mathrm{e}^{\xi(\varphi(t))} \Theta(\varphi(t)), \quad t\le I_{\varsigma}:= \int_0^{\varsigma} \mathrm{e}^{\alpha\xi(s)}\mathrm{d}s,
\end{equation}
where
\[
\varphi(t) := \inf\left\{s>0, \; \int_0^s \mathrm{e}^{\alpha \xi(u)}\mathrm{d}u > t \right\},
\]
and $I_{\varsigma} $ is the lifetime of $X$. Conversely, any process $X$ satisfying \eqref{eq: Lamperti R^d} is a self-similar Markov process in $\R^d$ with index $\alpha$.
\end{Prop}
\noindent In the previous statement we implicitly took the convention that $0\times \dagger = 0$. The integral $\zeta = I_{\varsigma}$ is the lifetime of $X$ until it eventually hits $0$, which acts as an absorbing state.

The analysis of MAPs with uncountable state space is much more intricate than the countable case. One way to capture their properties is using the compensation formula. It was proved in \cite{Cin} that any MAP $(\xi,\Theta)$ on $\R\times \Sb^{d-1}$ is associated with a so-called \emph{Lévy system} $(H,L)$, made up of an increasing additive functional $t\mapsto H_t$ of $\Theta$ and a transition kernel $L$ from $\Sb^{d-1}$ to $\R^*\times \Sb^{d-1}$, with $\R^*=\R\setminus\{0\}$, such that, for all $\theta\in\Sb^{d-1}$,
\[
\int_{\R^*} (1 \wedge |x|^2) \, L_{\theta}(\mathrm{d}x \times \{\theta\}) <\infty.
\]
More importantly, this Lévy system satisfies the following \emph{compensation formula} for all bounded measurable $F:(0,\infty)\times \R^2\times \Sb^{d-1}\times \Sb^{d-1} \rightarrow \R$, and all $(x,\theta)\in \R\times \Sb^{d-1}$,
\begin{multline} \label{eq: compensation MAP}
\mathtt{E}_{x,\theta}\left[\sum_{s>0} F(s,\xi(s^-),\Delta\xi(s),\Theta(s^-),\Theta(s))\mathds{1}_{\{\xi(s^-)\ne\xi(s) \, \text{or}\, \Theta(s^-)\ne\Theta(s)\}} \right]  \\
= \mathtt{E}_{x,\theta} \left[ \int_0^{\infty} \mathrm{d}H_s \int_{\R^*\times \Sb^{d-1}} L_{\Theta(s)}(\mathrm{d}x,\mathrm{d}\Phi) F(s,\xi(s),x,\Theta(s),\Phi) \right].
\end{multline}
For the remainder of the paper we restrict ourselves to the usual setting $H_t = t$. Because of the bijection in Proposition \ref{prop: Lamperti R^d}, this naturally puts us in a restricted class of self-similar Markov processes through the underlying driving MAP. Observe how \eqref{eq: compensation MAP} compares with the compensation formula for Lévy processes: $L$ essentially plays the role of a Lévy measure, albeit now depending on the current angle from which the process jumps.

A nice subclass of MAPs is provided by \emph{isotropic} self-similar Markov processes, and we shall mainly restrict ourselves to this setting. We say that a self-similar Markov process $X$ is isotropic if, for all isometry $U$, and all $\xbf\in\R^d\setminus\{0\}$, the law of $(U\cdot X(t), \Pb_{\xbf})$ is $\Pb_{U\cdot \xbf}$. Equivalently, this means \cite[Theorem 11.14]{KP} that for all $(x,\theta)\in\R\times \Sb^{d-1}$, the law of $((\xi,U\cdot \Theta), \mathtt{P}_{x,\theta})$ is $\mathtt{P}_{x, U\cdot \theta}$. The key advantage of restricting to isotropic processes is the following proposition, which is \cite[Corollary 11.15]{KP}.
\begin{Prop} \label{prop:isotropy Levy}
If $X$ is an isotropic self-similar Markov process, then the underlying ordinate $\xi$ is a Lévy process.
\end{Prop}
\noindent Let us briefly mention that the proof of \cref{prop:isotropy Levy} relies on the fact that by isotropy, $|X|$ is a positive self-similar Markov process, for which we can apply the classical Lamperti theory. This result opens the way to many useful Lévy tools, such as the Lévy-Itô description of $\xi$, the compensation or exponential formulas, or the existence of an exponential martingale and the corresponding change of measures. We will make heavy use of these additional properties when describing growth-fragmentations driven by isotropic processes in \cref{sec: spatial GF}. Note that this notion of isotropy in particular covers the $\alpha$--stable isotropic Lévy case \cite[Theorem 3.13]{Kyp}, for which the Lévy system is given by $H_t=t$ and 
\[
L_{\theta}(\mathrm{d}x,\mathrm{d}\Phi) =  \frac{c(\alpha)\mathrm{e}^{dx}}{|\mathrm{e}^x\Phi-\theta|^{\alpha+d}} \mathrm{d}x\sigma_{d-1}(\mathrm{d}\Phi),
\]
where $c(\alpha) = 2^{\alpha-1} \pi^{-d} \frac{\Gamma((d+\alpha)/2)\Gamma(d/2)}{|\Gamma(-\alpha/2)|}$, and $\sigma_{d-1}(\mathrm{d}\Phi)$ is the surface measure on the sphere $\Sb^{d-1}$. See also \cite{BW} for the planar case. Numerous applications of Lévy systems can be found in \cite{KRS,KRSY} to name but a few.

%-------------------------------------------------------------------------%
%               SPATIAL GROWTH-FRAGMENTATION PROCESSES 
%-------------------------------------------------------------------------%
\section{Spatial isotropic growth-fragmentation processes} \label{sec: spatial GF}
Now, we extend the framework of \cite{Ber-GF} and \cite{DP23} to isotropic $\R^d$--valued Markov processes for $d\ge 2$. We exclude the case $d=1$, since it can be deduced from the construction \cite{DP23} by considering a symmetric self-similar Markov process or  from  \cite{DS}. It is important to note that the construction  in \cite{DP23} does not consider the isotropy assumption as well as in \cite{DS}. 

In what follows $X$ will be an isotropic $\R^d$--valued self-similar Markov process with index $\alpha$, as defined in the last paragraph of \cref{sec: ssmp}, which under $\Pb_{\xbf}$, $\xbf\in \R^{d}\setminus \{0\}$, starts from $\xbf$. For technical reasons, we shall assume that $X$ is either absorbed after time $\zeta$ at some cemetery state $\partial$, or that $X$ converges to $0$ at infinity, for all starting points. We also recall that $(\xi,\Theta)$ denotes  the MAP associated with $X$.

%-------------------------------------------------------------------------%
%               SPATIAL GROWTH-FRAGMENTATION PROCESSES 
%-------------------------------------------------------------------------%
\subsection{Construction of spatial growth-fragmentation processes}\label{construction}
We now construct a cell system whose building block is the isotropic self-similar Markov process  $X$. The construction of the cell system in this case is similar to  \cite{Ber-GF} and simpler than the multitype case in  \cite{DP23}. It is important to note that  actually the construction holds without the self-similarity or isotropy assumptions. This cell system will start from a single particle whose size  is given by the process $X$, that will split in a binary way whenever $X$ has a jump.
 Let $\Delta X(t) := X(t)-X(t^-)$, for $t\ge 0$,  denotes the possible jump of $X$ at time $t$. At any jump time $t$ of $X$, one places a new particle in the system and, conditionally given their size $-\Delta X(t)$ at birth, each of these newborn particles evolves independently as $\Pb_{-\Delta X(t)}$. Then, one repeats this construction for any such child, thus creating the second generation, and so on.

We may now construct the cell system associated with $X$ and indexed by the tree $\Ub:=\bigcup_{i\ge 0} \N^i$, with $\N=\{1,2,\ldots\}$ and $\N^0:=\{\varnothing\}$ is the label of the \emph{Eve cell}. For $u:=(u_1,\ldots,u_i)\in \Ub$, we denote by $|u|=i$, the \emph{generation} of $u$. In this tree, the offspring of $u$ will be labelled by the lists $(u_1,\ldots,u_i,k)$, with $k\in \N$. 
 Let $b_{\varnothing}=0$ and $\Xcal_{\varnothing}$ be distributed as $X$ started from some vector $\xbf\in\R^d\setminus\{0\}$. At each jump of $\Xcal_{\varnothing}$, place a new particle with initial size given by minus the jump size (so that there is conservation at splitting events). Since $X$ converges at infinity, it is possible to rank the sequence of jump sizes and times $(\xbf_1,\beta_1),(\xbf_2,\beta_2),\ldots$ of $-\Xcal_{\varnothing}$ by descending lexicographical order for the norm of the $x_i$'s. Given this sequence of jumps, we define the first generation $\Xcal_i, i\in \N,$ of our cell system as independent processes with respective law $\Pb_{\xbf_i}$, and we set $b_i = b_{\varnothing}+\beta_i$ for the \emph{birth time} of $i$ and $\zeta_i$ for its lifetime. The law of the $n$-th generation is constructed likewise given generations $1,\ldots, (n-1)$. A cell $u'=(u_1,\ldots,u_{n-1})\in \N^{n-1}$ gives birth to the cell $u=(u_1,\ldots,u_{n-1},i)$, with lifetime $\zeta_u$, at time $b_u = b_{u'}+\beta_{i}$ where $\beta_{i}$ is the $i$-th jump of $\Xcal_{u'}$ (with respect to the previous ranking). Moreover, conditionally on the jump sizes and times of $\Xcal_{u'}$, $\Xcal_u$ has law $\Pb_{\mathbf{y}}$ with $-\mathbf{y}=\Delta \Xcal_{u'}(\beta_i)$ and is independent of the other daughter cells at generation $n$.

In this construction, the cells are not labelled chronologically. However, it still uniquely defines the law $\Pcal_{\xbf}$ of the cell system $(\Xcal_u(t), u\in\Ub, t\ge 0)$ started from $\xbf$. Finally, we introduce the \emph{(spatial) growth-fragmentation process}
\[
\Xbf(t) := \left\{\left\{ \Xcal_u(t-b_u), \; u\in\Ub \; \text{and} \; b_u\le t<b_u+\zeta_u \right\}\right\}, \quad t\ge 0,
\]
describing the collection of cells alive at time $t\ge 0$ (the double brackets here denote multisets). We define $\Pbf_{\xbf}$ to be the law of the growth-fragmentation $\Xbf$ started at $\xbf$.

We point out  that one can view this construction as a \emph{multitype growth-fragmentation} process, where the types correspond to the directions (in the $d=1$ case, it is the sign). The set of types is therefore the sphere $\Sb^{d-1}$, which is uncountable, so that the construction does not quite fall into the framework developed in \cite{DP23}. From this standpoint, note that the type corresponding to the daughter cell created by the jump $\Delta X(t)$ is, up to time-change,
\[
\Theta_{\Delta}(t) := \frac{\Theta(t^-)-\mathrm{e}^{\Delta \xi(t)}\Theta(t)}{|\Theta(t^-)-\mathrm{e}^{\Delta \xi(t)}\Theta(t)|}.
\]
Next,  let 
\[
\overline{\Xbf}(t) := \{\{(\Xcal_u(t-b_u),|u|), \; u\in\Ub \; \text{and} \; b_u\le t < b_u+\zeta_u\}\}, \quad t\ge 0.
\]
We shall denote by $(\Fcal_t, t\ge 0)$ the natural filtration associated with $\Xbf$, and $(\overline{\Fcal}_t, t\ge 0)$ the one associated with $\overline{\Xbf}$. Under the existence  of an excessive function   for $\Xbf$, 
that is that there exist a measurable function $f:\mathbb{R}^d\to [0,\infty)$
is called excessive for $\Xbf$ if 
\[
\mathbb{E}_x\left[\sum_{z\in \Xbf(t) }f(z)\right]\le f(x),
\]
for all $t\ge 0$ and $x\in \mathbb{R}\setminus \{0\}$. If such an excessive function exists one can rank the elements $X_1(t), X_2(t), \ldots$ of $\Xbf(t)$ by descending order of their norm for any fixed $t$. Under the same assumption, we have the following.
\begin{Prop} \label{prop: temporal branching spatial GF}
Assume that $\Xbf$ has an excessive function. Then for any $t\ge 0$, conditionally on $\overline{\Xbf}(t)=\{\{(\xbf_i,n_i)\}\}$, the process $(\overline{\Xbf}(t+s), s\ge 0)$ is independent of $\overline{\Fcal}_t$ and distributed as 
\[
\bigsqcup_{i\ge 1} \overline{\Xbf}_i(s) \circ \theta_{n_i},
\]
where the $\Xbf_i, i\ge 1,$ are independent processes distributed as $\overline{\Xbf}$ under $\Pcal_{\xbf_i}$, $\theta_n$ is the shift operator $\{\{(\mathbf{z}_i, k_i), i\ge 1\}\}\circ \theta_n := \{\{(\mathbf{z}_i, k_i+n), i\ge 1\}\}$, and $\sqcup$ denotes union of multisets.
\end{Prop}

We omit the proof  of the precedent result since it follows from the same arguments as in Proposition 2 in \cite{Ber-GF}. 

%-------------------------------------------------------------------------%
%               THE ISOTROPIC CUMULANT FUNCTIONS
%-------------------------------------------------------------------------%
\subsection{The isotropic cumulant function and genealogical martingales} \label{sec: cumulant}
We are first of all interested in pointing out martingales as in the monotype \cite{Ber-GF} and multitype \cite{DP23} setting  in the spatial growth-fragmentation case. It turns out that the exponents $\omega$ corresponding to these martingales will be found as the roots of an \emph{isotropic cumulant function} which generalises the cumulant function $\kappa$ in \cite{Ber-GF,BBCK}. Recall that, as readily seen from the rotational invariance property, the radial part of $X$ is a positive self-similar Markov process, so that the ordinate $\xi$ is in fact a Lévy process. We will extensively make use of this argument and its consequences.

Let us start with a simple but typical calculation: for $q\ge 0$ and $\theta\in\Sb^{d-1}$, we aim at computing the quantity $\Eb_{\theta}\Big[ \sum_{0<t<\zeta} |\Delta X(t)|^{q}\Big]$ in terms of the MAP characteristics of $X$.  We will now consider the Lévy system $(L,H)$ of $(\xi,\Theta)$ (see \cref{sec: ssmp}), and we take as usual $H_t=t$ to avoid notational clutter. Since we want to sum over all $t$'s, we can omit the Lamperti-Kiu time-change between $X$ and $(\xi,\Theta)$, so that
\[
\Eb_{\theta}\Bigg[ \sum_{0<t<\zeta} |\Delta X(t)|^{q}\Bigg]
=
\mathtt{E}_{0,\theta}\Bigg[ \sum_{0<t<\varsigma} \mathrm{e}^{q\xi(t^-)}|\Theta(t^-)-\mathrm{e}^{\Delta \xi(t)}\Theta(t)|^{q}\Bigg].
\]
The compensation formula \eqref{eq: compensation MAP} for Markov additive processes then yields
\begin{equation} \label{eq:compensation cumulant}
\Eb_{\theta}\Bigg[ \sum_{0<t<\zeta} |\Delta X(t)|^{q}\Bigg]
=
\mathtt{E}_{0,\theta}\left[ \int_{0}^{\infty} \mathrm{d}t \, \mathrm{e}^{q\xi(t)} \int_{\R^*\times \Sb^{d-1}} L_{\Theta(t)}(\mathrm{d}x, \mathrm{d}\Phi)|\Theta(t)-\mathrm{e}^{x}\Phi|^{q}\right].
\end{equation}
Remark that the integral
\[
\int_{\R^*\times \Sb^{d-1}} L_{\theta}(\mathrm{d}x, \mathrm{d}\Phi)|\theta-\mathrm{e}^{x}\Phi|^{q},
\]
does not depend on the angle $\theta$, since isotropy entails that if $\theta,\theta'\in\Sb^{d-1}$, and $U$ is an isometry mapping $\theta$ to $\theta'$, then $L_{\theta'}(\mathrm{d}x, \mathrm{d}\Phi) = L_{\theta}(\mathrm{d}x,U^{-1}\mathrm{d}\Phi)$. More generally, the image measures $\Ltilde_{\theta}(\mathrm{d}y,\mathrm{d}\phi)$ of $L_{\theta}(\mathrm{d}x,\mathrm{d}\Phi)$ through the mapping $(y,\phi) = (\log |\theta - \mathrm{e}^x\Phi|, \frac{\theta - \mathrm{e}^x\Phi}{|\theta - \mathrm{e}^x\Phi|})$ satisfy the same relationship. Indeed, for any nonnegative measurable function $F$, since $|U|=1$, we have
\begin{equation}\label{eq: Ltilde isotropy}
\begin{split}
\int_{\R^*\times \Sb^{d-1}} L_{\theta'}(\mathrm{d}x, \mathrm{d}\Phi)&F\left(\log|\theta'-\mathrm{e}^{x}\Phi|, \frac{\theta'-\mathrm{e}^{x}\Phi}{|\theta'-\mathrm{e}^{x}\Phi|}\right) \\
&=
\int_{\R^*\times \Sb^{d-1}} L_{\theta}(\mathrm{d}x, U^{-1}\mathrm{d}\Phi)F\left(\log|U\theta-\mathrm{e}^{x}\Phi|, \frac{U\theta-\mathrm{e}^{x}\Phi}{|U\theta-\mathrm{e}^{x}\Phi|}\right)   \\
&=
\int_{\R^*\times \Sb^{d-1}} L_{\theta}(\mathrm{d}x, \mathrm{d}\varphi)F\left(\log|U\theta-\mathrm{e}^{x}U\varphi|, \frac{U\theta-\mathrm{e}^{x}U\varphi}{|U\theta-\mathrm{e}^{x}U\varphi|}\right) \\
&=
\int_{\R^*\times \Sb^{d-1}} L_{\theta}(\mathrm{d}x, \mathrm{d}\varphi)F\left(\log|\theta-\mathrm{e}^{x}\varphi|, U \frac{\theta-\mathrm{e}^{x}\varphi}{|\theta-\mathrm{e}^{x}\varphi|}\right). 
\end{split}
\end{equation}
Hence $L_{\theta'}(\mathrm{d}x, \mathrm{d}\Phi) = L_{\theta}(\mathrm{d}x,U^{-1}\mathrm{d}\Phi)$. Singling out the image measure  $\Ltilde(\mathrm{d}y,\mathrm{d}\phi)$ of $L_{\theta}(\mathrm{d}x,\mathrm{d}\Phi)$ when $\theta=(1,0,\ldots,0)$ say, \eqref{eq:compensation cumulant} boils down to 
\[
\Eb_{\theta}\Bigg[ \sum_{0<t<\zeta} |\Delta X(t)|^{q}\Bigg]
=
\mathtt{E}_{0,\theta}\left[ \int_{0}^{\infty} \mathrm{d}t \, \mathrm{e}^{q\xi(t)}\right] \int_{\R^*\times \Sb^{d-1}} \Ltilde(\mathrm{d}y, \mathrm{d}\phi)\mathrm{e}^{qy}.
\]
Recall that  $\xi$ is a possibly killed Lévy process and assume that its Laplace exponent $\psi$ satisfies $\psi(q)<0$, otherwise the first integral blows up. Then we are left with
\[
\Eb_{\theta}\Bigg[ \sum_{0<t<\zeta} |\Delta X(t)|^{q}\Bigg]
=
1-\frac{\kappa(q)}{\psi(q)},
\]
where we have set
\begin{equation} \label{eq:kappa}
\kappa(q) = \psi(q) + \int_{\R^*\times \Sb^{d-1}} \Ltilde(\mathrm{d}y, \mathrm{d}\phi)\mathrm{e}^{qy}.
\end{equation}
\added{We stress once more that $\kappa$ can be calculated using any of the measures $\Ltilde_{\theta}$ in place of $\Ltilde$.} The previous calculations finally show that 
\begin{equation} \label{eq:sum kappa}
\Eb_{\theta}\Bigg[ \sum_{0<t<\zeta} |\Delta X(t)|^{q}\Bigg]
=
\begin{cases}
1-\displaystyle\frac{\kappa(q)}{\psi(q)} & \text{if} \; \kappa(q)<\infty \; \text{and} \; \psi(q)<0, \\[2mm]
+\infty & \text{otherwise}.
\end{cases}
\end{equation}
We call the function $\kappa$ the \emph{isotropic cumulant function}. Its roots will lead to martingales for the growth-fragmentation cell system through the identity \eqref{eq:sum kappa}. Thus, throughout the paper we make the following assumption
\begin{center}
\textbf{(H)} \quad \emph{There exists $\omega\ge 0$ such that $\kappa(\omega)=0$.}  
\end{center}
Notice that, as readily seen from \eqref{eq:kappa}, $\kappa$ is a convex function, so that there exist at most two such roots. For such a root $\omega$, we obtain by self-similarity and \eqref{eq:sum kappa} that for all $\xbf\in\R^d\setminus\{0\}$,
\begin{equation}\label{eq: omega generalised.0}
\Eb_{\xbf}\Bigg[ \sum_{0<t<\zeta} |\Delta X(t)|^{\omega}\Bigg]
=
|\xbf|^{\omega}.
 \end{equation}
Following the strategy of Section 3.2 in \cite{DS}, we now show that the roots of $\kappa$ pave the way for remarkable martingales.  The proof of the following result follows exactly from the same arguments as those used  in Proposition 3.6 in \cite{DS}. 
\begin{Prop}
Under $\Pb_{\xbf}$, for all $\xbf\in \R^d\setminus \{0\}$, the process 
\[
M(t) := |X(t)|^{\omega} + \sum_{0<s\le t\wedge \zeta}  |\Delta X(s)|^{\omega},
\]
is a martingale for the filtration $(F_t^X, t\ge 0)$ associated with $X$.
\end{Prop}

\medskip
\noindent Moreover, the definition of $\omega$ and the branching structure of growth-fragmentation processes entail the existence of the following genealogical martingale, which will be crucial for the spine decomposition. Let 
$\Gscr_n := \sigma\left(\Xcal_u, |u|\le n \right), n\ge 0$. 
\begin{Thm} \label{thm:M(n)}
The process
\[
\Mcal(n):= \sum_{|u|=n+1} |\Xcal_u(0)|^{\omega}, \quad n\ge 0,
\]
is a $(\Gscr_n, n\ge 0)$--martingale under $\Pcal_{\xbf}$ for all $\xbf\in \R^d\setminus\{0\}$.
\end{Thm}

The arguments used to deduce the previous result are the same as those presented in Theorem 3.5 in \cite{DS}.

%-------------------------------------------------------------------------%
%               A CHANGE OF MEASURES
%-------------------------------------------------------------------------%
\subsection{A change of measures}
We introduce a new probability measure $\Phat_{\xbf}$ for $\xbf\in \R^d \setminus\{ 0\}$ using the martingale $(\Mcal(n))_{n\ge 0}$ in \cref{thm:M(n)}. This is the analogue of \cite[Section 4.1]{BBCK} in the positive case or \cite[Section 3.3]{DS} in the $d=1$ case. It describes the law of a new cell system $(\Xcal_u)_{u\in\Ub}$ together with an infinite distinguished \emph{ray}, or \emph{leaf}, $\Lcal \in \partial \Ub = \N^{\N}$. On $\mathscr{G}_n$, for $n\ge 0$, it has Radon-Nikodym derivative with respect to $\Pcal_{\xbf}$ given by $\Mcal(n)$, normalized to be a probability measure, \emph{i.e.} for all $G_n \in \mathscr{G}_n$,
\[
\Phat_{\xbf}(G_n) := |\xbf|^{-\omega} \Ecal_{\xbf}\left[\Mcal(n) \mathds{1}_{G_n} \right].
\]
The law of the particular leaf $\Lcal$ under $\Phat_z$ is chosen so that, for all $n\ge 0$ and all $u\in \Ub$ such that $|u|=n+1$
\begin{equation} \label{eq: generation spine}
\Phat_{\xbf} \left( \Lcal(n+1)=u \,\big|\, \mathscr{G}_n\right) := \frac{|\Xcal_u(0)|^{\omega}}{\Mcal(n)},
\end{equation}
where for any $\ell\in\partial\Ub$, $\ell(n)$ denotes the ancestor of $\ell$ at generation $n$. In words, to define the next generation of the spine, we select one of its jumps proportionally to its size to the power $\omega$ (the spine at generation $0$ being the Eve cell). By an application of the Kolmogorov extension theorem, the martingale property and the branching structure of the system ensure that these definitions are compatible, and therefore this uniquely defines the probability measure $\Phat_{\xbf}$.

We will be interested in the evolution of the \emph{tagged cell}, which is the cell associated with the distinguished leaf $\Lcal$. More precisely, set $b_{\ell} = \lim\uparrow b_{\ell(n)}$ for any leaf $\ell\in\partial \Ub$. Then, define $\Xhat$ by $\Xhat(t):=\partial$ if $t\ge b_{\Lcal}$ and 
\begin{equation} \label{eq:tagged cell}
\Xhat(t):= \Xcal_{\Lcal(n_t)}(t-b_{\Lcal(n_t)}), \quad t<b_{\Lcal},
\end{equation}
where $n_t$ is the unique integer $n$ such that $b_{\Lcal(n)}\le t < b_{\Lcal(n+1)}$.

By construction of $\Phat_{\xbf}$, we have the following genealogical \emph{many-to-one} formula: for all nonnegative measurable function $f$ and all $\Gscr_n$--measurable nonnegative random variable $B_n$, 
\[
|\xbf|^{\omega} \Ehat_{\xbf}\left[f(\Xcal_{\Lcal(n+1)}(0))B_n\right]
=
\Ecal_{\xbf}\Bigg[ \sum_{|u|=n+1} |\Xcal_u(0)|^{\omega}f(\Xcal_u(0))B_n\Bigg].
\]
This may be extended to a temporal many-to-one formula. The existence of $(v,\omega)$ ensures that we may rank the elements in $\Xbf(t)=\left\{\left\{X_i(t), i\ge 1\right\}\right\}, \, t\ge 0$, by decreasing order of the norms.
\begin{Prop} \label{prop:spine temporal}
For every $t\ge 0$, every nonnegative measurable function $f$ vanishing at $\partial$, and every $\overline{\Fcal}_t$--measurable nonnegative random variable $B_t$, we have
\[
|\xbf|^{\omega} \Ehat_{\xbf}\big[f(\Xhat(t))B_t\big]
=
\Ecal_{\xbf}\Bigg[ \sum_{i\ge 1} |X_i(t)|^{\omega}f(X_i(t))B_t\Bigg].
\]
\end{Prop}
\begin{proof}
See Proposition 4.1 in \cite{DP23} for the multitype case, which is easily extended.
\end{proof}

%-------------------------------------------------------------------------%
%               THE ISOTROPIC CASE
%-------------------------------------------------------------------------%
\section{The spine decomposition of spatial isotropic growth-fragmentation processes}

\label{sec: spine isotropic}

%-------------------------------------------------------------------------%
%               THE SPINE DECOMPOSITION FOR ISOTROPIC GF
%-------------------------------------------------------------------------%
\subsection{The spine decomposition for isotropic growth-fragmentation processes}
In this section, we describe the law of the growth-fragmentation process under the change of measures $\Phat_{\xbf}$, $\xbf\in\R^d$, and in particular the law of the tagged cell $\Xhat$ \eqref{eq:tagged cell}. In order to make sense of this, we need to rebuild the growth-fragmentation along the spine, and so we must first label the jumps of $\Xhat$. In general, one cannot rank those in lexicographical order. Instead, they will be labelled by couples $(n,j)$, where $n\ge 0$ stands for the generation of the tagged cell immediately before the jump, and $j\ge 1$ is the rank (in the usual lexicographical sense) of the jump among those of the tagged cell at generation $n$ (including the final jump, when the generation changes to $n+1$). For each such $(n,j)$, we define the growth-fragmentation $\Xbfhat_{n,j}$ induced by the corresponding jump. More precisely, if the generation stays constant during the $(n,j)$--jump, then we set
\[
\Xbfhat_{n,j}(t) := \left\{\left\{ \Xcal_{uw}(t-b_{uw}+b_u), \; w\in\Ub \; \text{and} \; b_{uw}\le t+b_u<b_{uw}+\zeta_{uw} \right\}\right\},
\]
where $u$ is the label of the cell born at the $(n,j)$--jump. Otherwise the $(n,j)$--jump corresponds to a jump for the generation of the tagged cell so that the tagged cell jumps from label $u$ to label $uj$ say. In this case, we set
\begin{multline*} 
\Xbfhat_{n,j}(t) :=   \left\{\left\{ (\Xcal_{u}(t-b_{u}+b_{uj}), \Jcal_{u}(t-b_{u}+b_{uj})),  \;  b_{u}\le t+b_{uj}<b_{u}+\zeta_{u} \right\}\right\} \\ \cup \left\{\left\{ (\Xcal_{uw}(t-b_{uw}+b_{uj}), \Jcal_{uw}(t-b_{uw}+b_{uj})), \; w\notin \mathbb{T}_{uj} \; \text{and} \;  b_{uj}\le b_{uw}\le t+b_{uj}<b_{uw}+\zeta_{uw} \right\}\right\},
\end{multline*}
where for $v\in\Ub$, $\mathbb{T}_v := \{vw, \, w\in\Ub\}$.
Finally, we agree that $\Xbfhat_{n,j} := \partial$ when the $(n,j)$--jump does not exist, and this sets $\Xbfhat_{n,j}$ for all $n\ge 0$ and all $j\ge 1$.

Recall also that $n_t$ was defined in \eqref{eq:tagged cell} and stands for the generation of the spine at time $t$. We can now state our main theorem describing the law of the growth-fragmentation under $\Phat_{\xbf}$. 
\begin{Thm} \label{thm:spine}
Under $\Phat_{\xbf}$, $\Xhat$ is a self-similar Markov process with values in $\R^d$ and index $\alpha$. The Lévy system of the underlying Markov additive process $(\xihat,\Thetahat)$ is given by $(\Hhat,\Lhat)$ where $\Hhat_t=t$ and 
\begin{equation} \label{eq: Lhat}
\Lhat_{\theta}(\mathrm{d}y,\mathrm{d}\phi) = \mathrm{e}^{\omega y}\big(L_{\theta}(\mathrm{d}y,\mathrm{d}\phi) + \Ltilde_{\theta}(\mathrm{d}y,\mathrm{d}\phi) \big).
\end{equation}
Besides, $\Xhat$ is isotropic, and the ordinate $\xihat$ is a Lévy process with Laplace exponent $\psihat(q)=\kappa(\omega+q)$. Moreover, conditionally on $(\Xhat(t), n_t)_{0\le t<b_{\Lcal}}$, the processes $\Xbfhat_{n,j}$, $n\ge 0$, $j\ge 1$, are independent and each $\Xbfhat_{n,j}$ has law $\Pbf_{x(n,j)}$ where $-x(n,j)$ is the size of the $(n,j)$--th jump.
\end{Thm}
\begin{Rk} \label{rk: spine}
\begin{enumerate}
    \item Observe that we have the following description of the MAP $(\xihat,\Thetahat)$. Let $(\eta^0,\Phi^0)$ be a MAP with Lévy system $(H^0,L^0)$ given by $H^0_t:=t$ and $L^0_{\theta}(\mathrm{d}y,\mathrm{d}\phi):=\mathrm{e}^{\omega y} L_{\theta}(\mathrm{d}y,\mathrm{d}\phi)$. Consider an independent compound Poisson process $D=(D_1,D_2)$ on $\R^*\times \Sb^{d-1}$ with intensity measure $\mathrm{e}^{\omega y} \Ltilde(\mathrm{d}y,\mathrm{d}\phi)$. This definition makes sense because, since $\kappa(\omega)=0$, 
    \[
    \int_{\R^*\times \Sb^{d-1}}\Ltilde(\mathrm{d}y,\mathrm{d}\varphi)\mathrm{e}^{\omega y}
    =
    - \psi(\omega) <\infty.
    \]
    Then $(\xihat,\Thetahat)$ is the \emph{superimposition} of $(\eta^0,\Phi^0)$ and $D$, in the following sense. Let  $T_1$ the first jump time of $D$, which is exponential with parameter $-\psi(\omega)$. Then $(\xihat(s),\Thetahat(s), s<b_{\Lcal(1)})$ evolves as $(\eta^0(s),\Phi^0(s), s<T_1)$, and $(\xihat(b_{\Lcal(1)}),\Thetahat(b_{\Lcal(1)}))$ is distributed as 
    \[
    (\eta^0(T_1)+D_1(T_1), U_{\Phi^0(T_1)}\cdot D_2(T_1)),
    \]
     where $U_{\theta}$ is an isommetry mapping $(1,0,\ldots,0)$ to $\theta$.
    \item The proof actually provides a more precise statement describing the law of $(\Xhat(t),n_t, t\ge 0)$. The process $n_t$ is then the Poisson process counting the jumps \added{arising in $D$} up to the usual Lamperti time change.
    \item \added{The MAP $(\xi^0,\eta^0)$ is exactly the so-called \emph{Esscher transform} $(\xiomega,\Thetaomega)$ of $(\xi,\Theta)$}. More precisely, recall that in the isotropic setting, $\xi$ is itself a Lévy process, so that we can consider the usual exponential martingale $(\mathrm{e}^{\omega\xi(t)-t\psi(\omega)},t\ge 0)$. Then the law of $(\xi,\Theta)$ under the exponential change of measures is $(\xiomega,\Thetaomega)$. This will appear in the proof.
    \item Combining these two remarks casts light on equation \eqref{eq: Lhat}. Loosely speaking, it is a decomposition of $\Lhat$ in terms of the jumps of the Esscher transform of $(\xi,\Theta)$ and the special jumps when the spine picks one of the jumps according to \eqref{eq: generation spine}.
    \item  We deduce from \cref{thm:spine} that the temporal version of $(\Mcal(n), n\ge 0)$, namely
    \[
    \Mcal_t := \sum_{i=1}^{\infty} |X_i(t)|^{\omega}, \quad t\ge 0,
    \]
    is a $(\Fcal_t)$--martingale if, and only if, $\alpha\kappa'(\omega)<0$. Indeed, by taking $f=\mathds{1}_{\partial}$ the many-to-one formula (\cref{prop:spine temporal}) yields that $(\Mcal_t,t\ge 0)$ is a supermartingale, and that it is a martingale if, and only if, $X$ has infinite lifetime. From the Lamperti representation of $|X|$, and the expression $\psihat(q)=\kappa(\omega+q)$ of the Laplace exponent of $\xihat$, this happens exactly when $\alpha\kappa'(\omega)<0$.
\end{enumerate}
\end{Rk}

%-------------------------------------------------------------------------%
%               THE SPINE DECOMPOSITION
%-------------------------------------------------------------------------%
\subsection{Proof of \cref{thm:spine}}
The proof will roughly follow the same lines as the one of Theorem 4.3 in \cite{DP23}, although the structure of the modulator is more involved.

\bigskip
\noindent \textbf{The law of the spine $\Xhat$.} The definition of $\Xhat$ readily shows that $\Xhat$ is an $\alpha$--self-similar Markov process. By Lamperti's time change, we may place ourselves in the homogeneous case $\alpha=0$. In this case, note that there is no time change between $\Xhat$ and $(\xihat,\Thetahat)$. For this reason, and to avoid notational clutter, we will sometimes make an abuse of notation by considering them on the same probability space. Likewise, we will use expressions involving both $X$  and its MAP $(\xi,\Theta)$ as a shorthand. Moreover, the Markov property implies that we only need to check the compensation formula up to the first time $b_{\Lcal(1)}$ when the spine selects another generation. 
More precisely, we want to show that
\begin{multline}
\Ehat_{\theta}\left[\sum_{s>0}  F(s,\xihat(s^-),\Delta\xihat(s),\Thetahat(s^-),\Thetahat(s)) \mathds{1}_{\{s\le b_{\Lcal(1)}\}} \right] \\
=
\Ehat_{\theta}\left[ \int_{0}^{\infty} \mathrm{d}s \mathrm{e}^{\psi(\omega)s} \int_{\R^*\times \Sb^{d-1}} \Lhat_{\Thetahat(s)}(\mathrm{d}x, \mathrm{d}\varphi) F(s,\xihat(s),x, \Thetahat(s), \varphi)\right]. 
\end{multline}
We may split the sum into two parts:
\begin{multline}
    \Ehat_{\theta}\left[\sum_{s>0}  F(s,\xihat(s^-),\Delta\xihat(s),\Thetahat(s^-),\Thetahat(s)) \mathds{1}_{\{s\le b_{\Lcal(1)}\}} \right] = \\
    \Ehat_{\theta}\left[\sum_{s<b_{\Lcal(1)}} F(s,\xihat(s^-),\Delta\xihat(s),\Thetahat(s^-),\Thetahat(s)) \right] + \Ehat_{\theta}\Big[ F(b_{\Lcal(1)},\xihat(b_{\Lcal(1)}^-),\Delta\xihat(b_{\Lcal(1)}),\Thetahat(b_{\Lcal(1)}^-),\Thetahat(b_{\Lcal(1)})) \Big].     \label{eq: sum decomposed b(1)} 
\end{multline}

We compute the first term of \eqref{eq: sum decomposed b(1)}. By definition of $b_{\Lcal(1)}$, 
\[
(\xihat(s),\Thetahat(s), s<b_{\Lcal(1)}) = (\xi(s),\Theta(s), s<b_{\Lcal(1)}).
\]
Applying the change of measure \eqref{eq: generation spine}, and recalling that we are in the homogeneous case, we get
\begin{multline}
\Ehat_{\theta}\left[\sum_{s<b_{\Lcal(1)}} F(s,\xihat(s^-),\Delta\xihat(s),\Thetahat(s^-),\Thetahat(s)) \right] \\
=
\Eb_{\theta}\left[\sum_{s>0}\sum_{t>s} F(s,\xi(s^-),\Delta\xi(s),\Theta(s^-),\Theta(s)) |\Delta X(t)|^{\omega}\right]. \label{eq: split b1}
\end{multline}
Now, the Markov property of $X$ at fixed time $s>0$ yields that 
\begin{multline*}
\Eb_{\theta}\left[\sum_{t>s} F(s,\xi(s^-),\Delta\xi(s),\Theta(s^-),\Theta(s)) |\Delta X(t)|^{\omega}\right] \\
=
\Eb_{\theta}\left[ F(s,\xi(s^-),\Delta\xi(s),\Theta(s^-),\Theta(s)) \Eb_{X(s)} \left[\sum_{t>0} |\Delta X(t)|^{\omega}\right]\right], 
\end{multline*}
and using the definition of $\omega$ in identity \eqref{eq: omega generalised.0},
\[
\Eb_{\theta}\left[\sum_{t>s} F(s,\xi(s^-),\Delta\xi(s),\Theta(s^-),\Theta(s)) |\Delta X(t)|^{\omega}\right]
=
\mathtt{E}_{0,\theta}\Big[ F(s,\xi(s^-),\Delta\xi(s),\Theta(s^-),\Theta(s)) \mathrm{e}^{\omega \xi(s)}\Big].
\]
Coming back to \eqref{eq: split b1}, this means 
\[
\Ehat_{\theta}\left[\sum_{s<b_{\Lcal(1)}} F(s,\xihat(s^-),\Delta\xihat(s),\Thetahat(s^-),\Thetahat(s)) \right] 
=
\mathtt{E}_{0,\theta}\left[\sum_{s>0} F(s,\xi(s^-),\Delta\xi(s),\Theta(s^-),\Theta(s)) \mathrm{e}^{\omega \xi(s)}\right]. 
\]
Using the compensation formula entails
\begin{multline}
\Ehat_{\theta}\left[\sum_{s<b_{\Lcal(1)}} F(s,\xihat(s^-),\Delta\xihat(s),\Thetahat(s^-),\Thetahat(s)) \right] \\
=
\mathtt{E}_{0,\theta}\left[\int_{0}^{\infty} \mathrm{d}s \mathrm{e}^{\omega \xi(s)} \int_{\R^*\times \Sb^{d-1}} L_{\Theta(s)}(\mathrm{d}x,\mathrm{d}\varphi) \mathrm{e}^{\omega x} F(s,\xi(s),x,\Theta(s),\varphi) \right]. \label{eq: compensation < b1}
\end{multline}
We now tilt the measure using the classical Esscher transform (see for example \cite{KP}). Recall from \cref{rk: spine} that the process obtained has the law $\mathtt{P}^0_{0,\theta}$ of $(\eta^0,\Phi^0)$. Thus equation \eqref{eq: compensation < b1} rewrites 
\begin{multline}
\Ehat_{\theta}\left[\sum_{s<b_{\Lcal(1)}} F(s,\xihat(s^-),\Delta\xihat(s),\Thetahat(s^-),\Thetahat(s)) \right] \\
=
\mathtt{E}^0_{0,\theta}\left[\int_{0}^{\infty} \mathrm{d}s  \mathrm{e}^{\psi(\omega)s}\int_{\R^*\times \Sb^{d-1}} L_{\Phi^0(s)}(\mathrm{d}x,\mathrm{d}\varphi) \mathrm{e}^{\omega x} F(s,\eta^0(s),x,\Phi^0(s),\varphi) \right]. \label{eq: first term b1}
\end{multline}
Note that, since $L_{\theta}(\mathrm{d}x,\mathrm{d}\varphi) \mathrm{e}^{\omega x}$ is the jump measure of the Lévy system associated with $(\eta^0,\Phi^0)$, this shows that $(\xihat(s),\Thetahat(s), s<b_{\Lcal(1)})$ behaves as $(\eta^0(s),\Phi^0(s), s<T_1)$, where $T_1$ is an independent exponential time with parameter $-\psi(\omega)$, a fact that could have been derived directly.

Let us now compute the second term of \eqref{eq: sum decomposed b(1)}. Changing the measure according to \eqref{eq: generation spine} again, one obtains
\begin{align*}
\Ehat_{\theta}&\Big[ F(b_{\Lcal(1)},\xihat(b_{\Lcal(1)}^-),\Delta\xihat(b_{\Lcal(1)}),\Thetahat(b_{\Lcal(1)}^-),\Thetahat(b_{\Lcal(1)})) \Big] \\
&=
\Eb_{\theta}\left[ \sum_{s>0} |\Delta X(s)|^{\omega} F\left(s,\xi(s^-),\log|\Delta X(s)|-\xi(s^-), \Theta(s^-), \frac{\Delta X(s)}{|\Delta X(s)|}\right)\right] \\
&=
\mathtt{E}_{0,\theta}\left[ \sum_{s>0} \mathrm{e}^{\omega\xi(s^-)}|\Theta(s^-)-\mathrm{e}^{\Delta \xi(s)}\Theta(s)|^{\omega} F(s,\xi(s^-),\log|\Theta(s^-)-\mathrm{e}^{\Delta \xi(s)}\Theta(s)|, \Theta(s^-), \Theta_{\Delta}(s))\right],
\end{align*}
where as usual 
\[
\Theta_{\Delta}(s) = \frac{\Theta(s^-)-\mathrm{e}^{\Delta \xi(s)}\Theta(s)}{|\Theta(s^-)-\mathrm{e}^{\Delta \xi(s)}\Theta(s)|}.
\]
 Using the compensation formula for $(\xi,\Theta)$, this is
\begin{multline*}
    \Ehat_{\theta}\Big[ F(b_{\Lcal(1)},\xihat(b_{\Lcal(1)}^-),\Delta\xihat(b_{\Lcal(1)}),\Thetahat(b_{\Lcal(1)}^-),\Thetahat(b_{\Lcal(1)})) \Big] \\
    =
    \mathtt{E}_{0,\theta}\left[ \int_{0}^{\infty} \mathrm{d}s \mathrm{e}^{\omega\xi(s)} \int L_{\Theta(s)}(\mathrm{d}x, \mathrm{d}\varphi) |\Theta(s)-\mathrm{e}^x\varphi|^{\omega}\right . \\
    \left. \times F\left(s,\xi(s),\log|\Theta(s)-\mathrm{e}^{x}\varphi|, \Theta(s), \frac{\Theta(s)-\mathrm{e}^x\varphi}{|\Theta(s)-\mathrm{e}^x\varphi|}\right)\right]
\end{multline*}
We want to perform the change of variables $(y,\phi)=(\log |\theta-\mathrm{e}^x\varphi|,\frac{\theta-\mathrm{e}^x\varphi}{|\theta-\mathrm{e}^x\varphi|})$ for fixed $\theta$ in the second integral. Recall that we have defined \added{$\Ltilde_{\theta}$} as the image measure of $L_{\theta}$ through this mapping, and that \added{these measures satisfy the isotropy relationship \eqref{eq: Ltilde isotropy}}. Therefore,
\begin{multline*}
    \Ehat_{\theta}\Big[ F(b_{\Lcal(1)},\xihat(b_{\Lcal(1)}^-),\Delta\xihat(b_{\Lcal(1)}),\Thetahat(b_{\Lcal(1)}^-),\Thetahat(b_{\Lcal(1)})) \Big] \\
    =
    \mathtt{E}_{0,\theta}\left[ \int_{0}^{\infty} \mathrm{d}s \mathrm{e}^{\omega\xi(s)} \int_{\R^*\times \Sb^{d-1}} \Ltilde_{\added{\Theta(s)}}(\mathrm{d}y, \mathrm{d}\phi) \mathrm{e}^{\omega y} F(s,\xi(s),y, \Theta(s), \phi)\right].
\end{multline*}
Tilting with the exponential martingale of $\xi$ finally provides
\begin{multline}
    \Ehat_{\theta}\Big[ F(b_{\Lcal(1)},\xihat(b_{\Lcal(1)}^-),\Delta\xihat(b_{\Lcal(1)}),\Thetahat(b_{\Lcal(1)}^-),\Thetahat(b_{\Lcal(1)})) \Big] \\
    =
    {\tt E}^0_{0,\theta}\left[ \int_{0}^{\infty} \mathrm{d}s \mathrm{e}^{\psi(\omega)s} \int_{\R^*\times \Sb^{d-1}} \Ltilde_{\added{\Phi^0(s)}}(\mathrm{d}y, \mathrm{d}\phi) \mathrm{e}^{\omega y} F(s,\added{\eta^0(s)},y, \added{\Phi^0(s)}, \phi)\right]. \label{eq: second term b1}
\end{multline}
Putting together \eqref{eq: sum decomposed b(1)}, \eqref{eq: first term b1} and \eqref{eq: second term b1}, we end up with
\begin{multline*}
\Ehat_{\theta}\left[\sum_{s>0}  F(s,\xihat(s^-),\Delta\xihat(s),\Thetahat(s^-),\Thetahat(s)) \mathds{1}_{\{s\le b_{\Lcal(1)}\}} \right] \\
=
{\tt  E}^0_{0,\theta}\left[ \int_{0}^{\infty} \mathrm{d}s \mathrm{e}^{\psi(\omega)s} \int_{\R^*\times \Sb^{d-1}} \Lhat_{\Phi^0(s)}(\mathrm{d}x, \mathrm{d}\varphi) F(s,\eta^0(s),x, \Phi^0(s), \varphi)\right],
\end{multline*}
and \added{since $(\eta^0(s),\Phi^0(s),s<T_1)$ has the same law as $(\xihat(s),\Thetahat(s), s<b_{\Lcal(1)})$}, we can rewrite this as
\begin{multline}
\Ehat_{\theta}\left[\sum_{s>0}  F(s,\xihat(s^-),\Delta\xihat(s),\Thetahat(s^-),\Thetahat(s)) \mathds{1}_{\{s\le b_{\Lcal(1)}\}} \right] \\
=
\Ehat_{\theta}\left[ \int_{0}^{\infty} \mathrm{d}s \mathrm{e}^{\psi(\omega)s} \int_{\R^*\times \Sb^{d-1}} \Lhat_{\Thetahat(s)}(\mathrm{d}x, \mathrm{d}\varphi) F(s,\xihat(s),x, \Thetahat(s), \varphi)\right]. \label{eq: final (xihat,Thetahat)}
\end{multline}
This completes the proof of \eqref{eq: Lhat}.

The second assertion of the theorem is then a straightforward consequence. First, it is clear that since $X$ is isotropic, so is $\Xhat$ by construction. Hence, by \cref{prop:isotropy Levy}, $\xihat$ must be a Lévy process. The expression for $\psihat$ can be found using a particular case of the compensation formula \eqref{eq: final (xihat,Thetahat)}. Alternatively, for any nonnegative measurable functionals $F$ and $G$ defined respectively on the space of finite càdlàg paths and on $\R$, we may compute
\begin{align*}
    &\Ehat_{\theta}\left[F(\xihat(s),s<b_{\Lcal(1)})G(\Delta \xihat(b_{\Lcal(1)})) \right] \\
    &=
    \Eb_{\theta} \left[ \sum_{t>0} |\Delta X(t)|^{\omega} F(\log|X(s)|,s<t)G\left(\log\frac{|\Delta X(t)|}{|X(t^-)|}\right)\right] \\
    &=
    \mathtt{E}_{0,\theta} \left[ \sum_{t>0} \mathrm{e}^{\omega \xi(t^-)}|\Theta(t^-)-\mathrm{e}^{\Delta \xi(t)} \Theta(t)|^{\omega} F(\xi(s),s<t)G\Big(\log |\Theta(t^-)-\mathrm{e}^{\Delta\xi(t)}\Theta(t)|\Big)\right] \\
    &=
    \mathtt{E}_{0,\theta}\left[\int_0^{\infty} \mathrm{d}t \mathrm{e}^{\omega \xi(t)} F(\xi(s),s<t)\int_{\R^*\times \Sb^{d-1}}  L_{\Theta(t)}(\mathrm{d}x,\mathrm{d}\varphi) |\Theta(t)-\mathrm{e}^{x}\varphi|^{\omega} G(\log |\Theta(t)-\mathrm{e}^x\varphi|)\right].
\end{align*}
By isotropy of $X$, the second integral does not depend on the angle $\Theta(t)$ (see \eqref{eq: Ltilde isotropy}). Hence by applying the change of variables $(y,\phi)=\Big(\log|\Theta(t)-\mathrm{e}^{x}\varphi|, \frac{\Theta(t)-\mathrm{e}^{x}\varphi}{|\Theta(t)-\mathrm{e}^{x}\varphi|}\Big)$, we end up with 
\[
\begin{split}
\Ehat_{\theta}&\left[F(\xihat(s),s<b_{\Lcal(1)})G(\Delta \xihat(b_{\Lcal(1)})) \right]\\
&\hspace{2cm}=
\mathtt{E}_{0,\theta}\left[\int_0^{\infty} \mathrm{d}t \mathrm{e}^{\omega \xi(t)} F(\xi(s),s<t)\right]\int_{\R^*\times \Sb^{d-1}}  \Ltilde(\mathrm{d}y,\mathrm{d}\phi) \mathrm{e}^{\omega y} G(y).
\end{split}
\]
In words, this proves that $(\xihat(s),s<b_{\Lcal(1)})$ and $\Delta \xihat(b_{\Lcal(1)})$ are independent. The former has the law of $\xi$ killed according to its exponential martingale, leading to a Lévy process with Laplace exponent $q\mapsto \psi(\omega+q)$. On the other hand, the latter is distributed as $(\psi(\omega))^{-1} \int_{\phi\in \Sb^{d-1}} \Ltilde(\mathrm{d}y,\mathrm{d}\phi) \mathrm{e}^{\omega y}$, which is the law of the first jump of a compound Poisson process with intensity measure $\int_{\phi\in \Sb^{d-1}} \Ltilde(\mathrm{d}y,\mathrm{d}\phi) \mathrm{e}^{\omega y}$. By removing the killing, this entails that $\xihat$ has Laplace exponent
\[
\psihat(q) = \psi(\omega+q) - \psi(\omega) + \int_{\R^* \times \Sb^{d-1}} \Ltilde(\mathrm{d}y,\mathrm{d}\phi) \mathrm{e}^{\omega y} (\mathrm{e}^{qy}-1), \quad q\ge 0.
\]
Using that $\kappa(\omega)=0$, this is
\[
\psihat(q) = \psi(\omega+q) + \int_{\R^* \times \Sb^{d-1}} \Ltilde(\mathrm{d}y,\mathrm{d}\phi) \mathrm{e}^{(\omega+q) y}, \quad q\ge 0,
\]
whence $\psihat(q) = \kappa(\omega+q)$.

\bigskip
\noindent \textbf{The law of the growth-fragmentations $\Xbfhat_{n,j}$.}
We prove the last assertion of \cref{thm:spine}. It actually follows from the same arguments as in \cite{BBCK}, but we provide the proof for the sake of completeness. To avoid cumbersome notation, we will restrict to proving the statement for the first generation. This is then easily extended thanks to the branching property. Let $F$ be a nonnegative measurable functional on the space of càdlàg trajectories, and $G_j$, $j\ge 1$, be nonnegative measurable functionals on the space of multiset--valued paths. For $t>0$, denote by $(\Delta_j(t), j\ge 1)$ the sequence consisting of all the jumps of $\Xcal_{\varnothing}$ that happened strictly before time $t$, and the extra value of $\Xcal_{\varnothing}(t)$, all ranked in descending order of their absolute value. We are after the identity
\[
\Ehat_1 \left[F(\Xcal_{\varnothing}(s), 0\le s\le b_{\Lcal(1)}) \prod_{j\ge 1} G_j(\Xbfhat_{0,j})\right]
=
\Ehat_1 \left[F(\Xcal_{\varnothing}(s), 0\le s\le b_{\Lcal(1)}) \prod_{j\ge 1} \Ebf_{\Delta_j(b_{\Lcal(1)})} \left[G_j(\Xbf)\right]\right].
\]
We start from the left-hand side, and apply the change of measure \eqref{eq: generation spine}:
\begin{multline*}
\Ehat_1 \left[F(\Xcal_{\varnothing}(s), 0\le s\le b_{\Lcal(1)}) \prod_{j\ge 1} G_j(\Xbfhat_{0,j})\right] \\
=
\Ecal_1 \left[\sum_{t>0} |\Delta \Xcal_{\varnothing}(t)|^{\omega} F(\Xcal_{\varnothing}(s), 0\le s\le t) \prod_{j\ge 1} G_j(\Xbfhat_{0,j})\right].
\end{multline*}
Using the definition of the $\Xbfhat_{0,j}$ together with the branching property under $\Pcal_1$ give
\begin{multline*}
\Ehat_1 \left[F(\Xcal_{\varnothing}(s), 0\le s\le b_{\Lcal(1)}) \prod_{j\ge 1} G_j(\Xbfhat_{0,j})\right] \\
=
\Ecal_1 \left[\sum_{t>0} |\Delta \Xcal_{\varnothing}(t)|^{\omega} F(\Xcal_{\varnothing}(s), 0\le s\le t) \prod_{j\ge 1} \Ebf_{\Delta_j(t)} \left[G_j(\Xbf)\right]\right].
\end{multline*}
Applying the change of measure backwards, we get the desired identity. This concludes the proof of Theorem \ref{thm:spine}.

%-------------------------------------------------------------------------%
%               COMMENTS ISOTROPY
%-------------------------------------------------------------------------%
\subsection{Comments on the isotropy assumption} \label{sec:comments isotropy}
The previous analysis of $\R^d$--valued growth-fragmentations relies heavily on the isotropy assumption. Because of the complications caused by the underlying MAP structure, describing growth-fragmentations driven by anisotropic processes is a much more challenging task. We stress the importance of the isotropy assumption and comment on possible extensions to anisotropic growth-fragmentation processes. 

First, we expect that in the anisotropic case, there should be an angular component in all the (super-)martingales, appearing in particular in \cref{thm:M(n)}. This already takes place in the $d=1$ case \cite{DS}, for asymmetric signed growth-fragmentation processes, where the angular component is nothing but the sign. Remember in addition that, in analogy with the discrete multitype case \cite{DP23}, the \emph{types} in the spatial framework are the angles, and that the martingales in the multitype setting also involve the types, see Section 3.2 in \cite{DP23}. If $X$ is a $\R^d\setminus\{0\}$--valued self-similar Markov process, this actually prompts us to define, for $q\ge 0$, the linear operator 
\[ 
T_q: f\in \Ccal \mapsto \left( \theta\in\Sb^{d-1} \mapsto \Eb_{\theta}\left[\sum_{t>0} f(\Theta_{\Delta}(\varphi(t)))|\Delta X(t)|^q\right] \right),
\]
where $\varphi$ is the Lamperti-Kiu time-change. This is the analogue of the matrix $m$ appearing in the multitype case. Assume that $X$ has jumps (otherwise the construction is irrelevant), and that $M_q:=\underset{\theta\in\Sb^{d-1}}{\sup}\Eb_{\theta} \left[ \sum_{t>0} |\Delta X(t)|^q\right] <\infty$. Then $T_q(f)$ is well-defined for all $f\in \Ccal$, and for $f\in\Ccal$,
\[
|| T_q(f)||_{\infty} \le M_q || f ||_{\infty},
\]
whence $T_q$ is a continuous operator. Note also that, at least under the assumption that $X$ jumps with positive probability to any open set $D\subset\Sb^{d-1}$ of directions, $T_q$ is strongly positive, in the sense that for all nonnegative $f\ne 0$, $T_q(f)>0$. Assume moreover that $T_q$ takes values in $\Ccal$, and that it is a compact operator. Then, by the Krein-Rutman theorem \cite{Dei}, it must have positive spectral radius $r(q)$, which is moreover a simple eigenvalue associated to a positive eigenfunction $v$. In the spirit of Assumption \textbf{(H)}, \cref{sec: cumulant}, one could impose the additional assumption
\begin{center}
\textit{(H')} \quad \emph{There exists $\omega\ge 0$ such that $r(\omega)=1$.}  
\end{center}
Then by definition, we have 
\[
\forall \theta\in\Sb^{d-1}, \quad \Eb_{\theta}\left[ \sum_{t>0} v(\Theta_{\Delta}(\varphi(t)))|\Delta X(t)|^{\omega}\right]
=
v(\theta).
\]
This generalises to vectors in $\R^d$ by self-similarity of $X$:
\begin{equation} \label{eq: omega generalised}
\forall (r,\theta)\in \R_+ \times \Sb^{d-1}, \quad \Eb_{r\theta}\left[ \sum_{t>0} v(\Theta_{\Delta}(\varphi(t)))|\Delta X(t)|^{\omega}\right]
=
v(\theta) r^{\omega}.
\end{equation}
\begin{Rk} \label{rk:isotropic}
When $X$ is \emph{isotropic} in the sense of \cref{sec: ssmp}, one can show that $v(\theta)=1$ for all $\theta\in\Sb^{d-1}$ up to normalisation, and one therefore retrieves the cumulant approach presented in \cref{sec: cumulant}. Indeed, isotropy entails that if $v$ is an eigenfunction associated with $r(q)$, then for all isometries $U$, $v(U\cdot)$ is also an eigenfunction associated with $r(q)$, and we conclude by simplicity of the eigenvalue that $v(U\cdot) = v$, so that $v$ is constant. 
\end{Rk}

Once \eqref{eq: omega generalised} holds for some positive function $v$, then modulo these adjustments one can carry through the arguments for the genealogical martingale (\cref{thm:M(n)}) and the many-to-one formula (\cref{prop:spine temporal}). However, the description of the spine in \cref{thm:spine} is more involved. This is mainly due to the fact that the jump intensity at time $b_{\Lcal(1)}$ depends on the current angle of the spine. In the isotropic case, one can more or less get rid of this dependency. The proof of \cref{thm:spine} hinges upon the existence of an \emph{Esscher transform}. In the isotropic case, this readily comes from the fact that the ordinate $\xi$ of $X$ is a Lévy process, which does not hold anymore for anisotropic processes. This in particular yielded that $b_{\Lcal(1)}$ (up to Lamperti time change) is an exponential random variable. This last feature should not hold in general, as already indicated by the discrete multitype case.

%-------------------------------------------------------------------------%
%               THE GROWTH-FRAGMENTATION IN HALF-SPACE EXCURSIONS
%-------------------------------------------------------------------------%
\section{The growth-fragmentation embedded in Brownian excursions from hyperplanes}
\label{sec: excursion}

%-------------------------------------------------------------------------%
%               THE EXCURSION MEASURE
%-------------------------------------------------------------------------%
\subsection{The excursion measure} \label{sec:excursion measure}

\textbf{Construction of the excursion measure $\n_+$.} We fix $d\ge 3$ and recall from \cite{Bur} how one may define the Brownian excursion measure from hyperplanes in $\R^d$. Let $(\Omega, \Fscr, (\Fscr_t)_{t\ge 0}, \Pb)$ a complete filtered probability space, on which is defined a $d$--dimensional Brownian motion $B^d$. We single out the last coordinate and write $B^d=(B^{d-1}, Z)$. Introduce the set $\Xscr$ of càdlàg functions $x$ defined on some finite time interval $[0,R(x)]$, and the set $\Xscr_0$ of such functions $x$ in $\Xscr$ that are continuous and vanish at $R(x)$. Moreover, we define
\[
U := \left\{u:=(x_1,\ldots, x_{d-1},z)\in \Xscr^{d-1}\times \Xscr_0, \; R(x_1)=\ldots=R(x_{d-1})=R(z) \; \text{and} \; u(0)=0 \right\}.
\]
For $u\in U$, we shall write $R(u)$ for the common value of the lifetimes. All these sets are equipped with their usual $\sigma$--fields. Finally, in order to study the excursions of $B^d$ from the hyperplane $\Hcal=\{x_d=0\}$, we introduce the local time $(\ell_s, s\ge 0)$ at $0$ of the Brownian motion $Z$, as well as its inverse $(\uptau_s, s\ge 0)$. More precisely, $\ell$ is normalised as
\[
\ell_s := \lim_{\varepsilon\to 0} \frac{1}{2\varepsilon} \int_0^s \mathds{1}_{\{|Z_r|\le \varepsilon\}} \mathrm{d}r, \quad s\ge 0.
\]

The \emph{excursion process} $(\e_s, s>0)$ of our interest is easily defined following the one-dimensional case (see \cite{RY}, Chapter XII), by
\begin{itemize}
    \item[(i)] if $\uptau_s-\uptau_{s^-}>0$, then 
    \[\e_s : r\mapsto \left(B^{d-1}_{r+\uptau_{s^-}}-B^{d-1}_{\uptau_{s^-}}, Z_{r+\uptau_{s^-}}\right), \quad r\leq \uptau_s-\uptau_{s^-},\]
    \item[(ii)] if $\uptau_s-\uptau_{s^-}=0$, then $\e_s = \partial$,
\end{itemize}
where $\partial$ is some cemetery state. The following proposition directly stems from the one-dimensional case.

\medskip
\begin{Prop} \label{prop:excursion measure}
The excursion process $(\e_s, s>0)$ is a $(\Fscr_{\uptau_s})_{s>0}$--Poisson point process of excursions in $U$. Its intensity measure is 
\[
 (\mathrm{d}u',\mathrm{d}z) := n(\mathrm{d}z) \Pb\Big((B^{d-1})^{R(z)}\in \mathrm{d}u'\Big),
\]
where $n$ denotes the one-dimensional Itô measure on $\Xscr_0$, and for any process $X$, and any time $T$, $X^T:=(X_t, \, t\in [0,T])$.
\end{Prop}

\medskip
\noindent We shall denote by $\n_+$ and $\n_-$ the restrictions of $\n$ to $U^+:=\{(u',z)\in U, \; z\ge 0\}$ and $U^-:=\{(u',z)\in U, \; z\le 0\}$ respectively.
In \cite{Bur}, excursion measures from hyperplanes in $\R^d$ are rather constructed using Bessel processes. More precisely, one first samples the duration of the excursion with density $r\mapsto (2\pi r^3)^{-1/2}\mathds{1}_{\{r\ge 0\}}$ with respect to Lebesgue measure, and then for the last coordinate, one samples a $3$--dimensional Bessel bridge from $0$ to $0$ over $[0,r]$. This is equivalent to $\n_+$ in our representation (up to a multiplicative factor) thanks to Itô's description of $n$, for which we refer again to \cite{RY}. We conclude this paragraph with the following Markov property under $\n_+$. We set $\Fcal_t = \sigma(u(s), \, 0\le s\le t)$.
\begin{Prop} \label{prop:Markov}
On the event that $T_a:= \inf\{0\le t\le R(u), \, z(t)=a\}<\infty$, the process $(u(T_a+t)-u(T_a), 0\le t\le R(u)-T_a)$ is independent of $\Fcal_{T_a}$ and is a $d$-dimensional Brownian motion stopped when hitting $\{x_d=-a\}$.
\end{Prop}

\medskip
\noindent \textbf{Disintegration of $\n_+$.} We now construct measures $\gamma_{\xbf}$, $\xbf\in\R^{d-1}$, for Brownian excursions from the hyperplane $\Hcal=\{x_d=0\}$ \emph{conditioned} on ending at $(v,0)$, by disintegrating $\n_+$ over its endpoint. Whenever $r\ge 0$ and $\xbf\in \R^{d-1}$, we write $\Pi_r$ for the law of a Bessel bridge from $0$ to $0$ over $[0,r]$, and $\Pbb{r}{0}{\xbf}$ for the law of a $(d-1)$--dimensional Brownian bridge from $0$ to $\xbf$ with duration $r$. See \cite{AD} for the  case $d=2$.

\begin{Prop} \label{prop:disintegration}
The following disintegration formula holds:
\[
\n_+ = \int_{\R^{d-1}\setminus \{0\}} \mathrm{d}\xbf \, \frac{\Gamma(\frac{d}{2})}{2\pi^{d/2}|\xbf|^{d}} \cdot \gamma_{\xbf},
\]
where $\gamma_{\xbf}$, $\xbf\in\R^{d-1}\setminus\{0\}$, are probability measures. In addition, for all $\xbf\in\R^{d-1}\setminus\{0\}$,
\[
\gamma_{\xbf} = \int_0^{\infty} \mathrm{d}r \frac{\mathrm{e}^{-\frac{1}{2r}}}{2^{\frac{d}{2}}\Gamma\left(\frac{d}{2}\right) r^{\frac{d}{2}+1}} \Pbb{r|\xbf|^2}{0}{\xbf} \otimes \Pi_{r|\xbf|^2}.
\]
\end{Prop}
\begin{proof}
The proposition follows from Theorem 3.3 in \cite{Bur}, but we rephrase it in our framework for completeness. Let $f:\Xscr^{d-1}\longrightarrow \R_+$ and $g:\Xscr_0\longrightarrow \R_+$ be two nonnegative measurable functions. Then by Proposition \ref{prop:excursion measure},
\[
\int_{U^+} f(u')g(z) \n_+(\mathrm{d}u', \mathrm{d}z) 
=
\int_{U^+} f(u')g(z) n_+(\mathrm{d}z) \Pb\Big((B^{d-1})^{R(z)}\in \mathrm{d}u'\Big).
\]
Then by Itô's description of $\n_+$ (see Chap. XII, Theorem 4.2 in \cite{RY}), we may split this integral over the duration $R(z)$:
\[
\int_{U^+} f(u')g(z) \n_+(\mathrm{d}u', \mathrm{d}z) 
=
\int_{0}^{\infty} \frac{\mathrm{d}r}{2\sqrt{2\pi r^3}} \Pi_r[g] \Eb[f((B^{d-1})^r)].
\]
We now condition on $B^{d-1}_r$, and we obtain
\[
\int_{U^+} f(u')g(z) \n_+(\mathrm{d}u', \mathrm{d}z) 
=
\int_{0}^{\infty} \frac{\mathrm{d}r}{2\sqrt{2\pi r^3}} \int_{\R^{d-1}} \mathrm{d}\xbf \frac{\mathrm{e}^{-\frac{|\xbf|^2}{2r}}}{(2\pi r)^{\frac{d-1}{2}}}  \Pi_r[g] \Ebb{r}{0}{\xbf}[f].
\]
Finally, we perform the change of variables $r\mapsto t = r/|x|^2$:
\[
\int_{U^+} f(u')g(z) \n_+(\mathrm{d}u', \mathrm{d}z) 
=
\int_{\R^{d-1}} \frac{\mathrm{d}\xbf}{|x|^d} \int_{0}^{\infty} \mathrm{d}t  \frac{\mathrm{e}^{-\frac{1}{2t}}}{2(2\pi)^{\frac{d}{2}} t^{\frac{d}{2}+1}}  \Pi_{t|\xbf|^2}[g] \Ebb{t|\xbf|^2}{0}{\xbf}[f].
\]
Since 
\[
\int_{0}^{\infty} \mathrm{d}t  \frac{\mathrm{e}^{-\frac{1}{2t}}}{2(2\pi)^{\frac{d}{2}} t^{\frac{d}{2}+1}}
=
\frac12 \pi^{-\frac{d}{2}}\Gamma\left(\frac{d}{2}\right),
\]
this gives that  $\gamma_{\xbf}$, for $\xbf\in\R^{d-1}\setminus\{0\}$, are probability measures, and the disintegration claim holds.
\end{proof}

\begin{figure} 
\bigskip
\begin{center}
\includegraphics[scale=0.8]{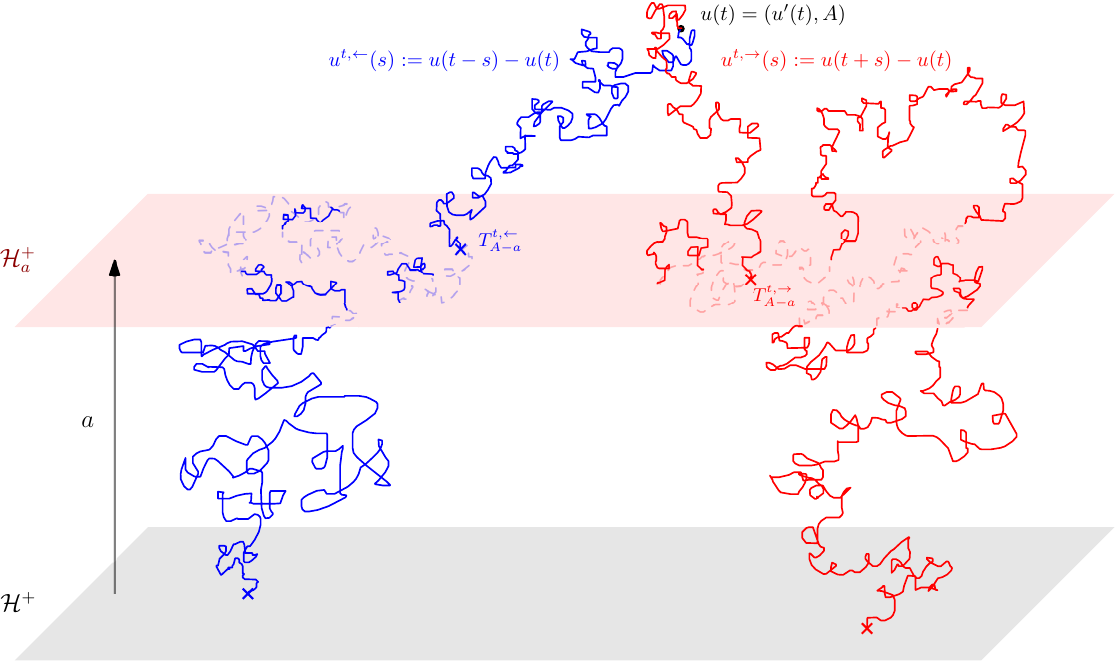}
\end{center}
\caption{Bismut's description of $\n_+$ in dimension $d=3$. The height of a uniformly chosen point $t$ on the excursion weighted by its duration is distributed according to the Lebesgue measure $\mathrm{d}A$. Moreover, conditionally on the height, the excursion splits into two independent trajectories depicted in blue and red. Both are distributed as Brownian motion killed when hitting the bottom half-plane (in grey).}
\label{fig:Bismut}
\end{figure}

\medskip
\noindent \textbf{Bismut's description of $\n_+$.} The following decomposition of $\n_+$ describes the left and right parts of the trajectory seen from a point chosen uniformly at random on the Brownian excursion weighted by its lifetime.
\begin{Prop} \label{prop:Bismut} (Bismut's description of $\n_+$)

\noindent Let $\overline{\n}_+$ be the measure defined on $\R_+\times U^+$ by
\[\overline{\n}_+(\mathrm{d}t,\mathrm{d}u) = \mathds{1}_{\{0\leq t\leq R(u)\}} \mathrm{d}t \, \n_+(\mathrm{d}u).\]
Then under $\overline{\n}_+$ the ``law'' of $(t,(u',z))\mapsto z(t)$ is the Lebesgue measure $\mathrm{d}A$ on $\R_+$, and conditionally on $z(t)=A$, $u^{t, \leftarrow}=\left(u(t-s)-u(t)\right)_{0\leq s\leq t}$ and $u^{t,\rightarrow}=\left(u(t+s)-u(t)\right)_{0\leq s\leq R(u)-t}$ are independent Brownian motions killed when reaching the hyperplane $\left\{x_d=-A\right\}$.
\end{Prop}

\noindent Proposition \ref{prop:Bismut} is a straightforward consequence of Bismut's description of the one-dimensional Itô measure $n$. \cref{fig:Bismut} illustrates how the excursion splits when seen from a uniform point.

%-------------------------------------------------------------------------%
%               CUTTING EXCURSIONS WITH HYPERPLANES
%-------------------------------------------------------------------------%

\subsection{Slicing excursions with hyperplanes} \label{sec: slicing excursions}
This section is an easy extension of the framework introduced in \cite{AD}. Let $u\in U^+$, and $a\ge 0$. We may write $u:=(u',z)$ with $u'\in \Xscr^{d-1}$ and $z\in \Xscr_0, z\ge 0$. 

\medskip
\noindent \textbf{Notation and setup.} 
Define the superlevel set 
\begin{equation} \label{eq:I(a) partition}
    \mathcal{I}(a) = \left\{s\in [0,R(u)], \; z(s)>a\right\}.
\end{equation} 
This is a countable (possibly empty) union of disjoint open intervals, and for any such interval $I=(i_-,i_+)$, we write $u_{I}(s) := u(i_- +s)-u(i_-), 0\leq s\leq i_+ -i_-,$ for the restriction of $u$ to $I$, and $\Delta u_I := x(i_+)-x(i_-)$. Remark that $\Delta u_I$ is a vector in the hyperplane $\Hcal_a :=\{x_d=a\}$, which we call the \emph{size} or \emph{length} of the excursion $u_I$, see \cref{fig:slicing excursion}. If $0\le t\le R(u)$, we denote by $e_a^{(t)}$ the excursion $u_I$ corresponding to the unique such interval $I$ which contains $t$. Moreover, we define $\Hcal_a^+$ as the set of excursions above $\Hcal_a$ corresponding to the previous partition of $\mathcal{I}(a)$.

We may now present an application of Proposition \ref{prop:Bismut}, which is similar to Proposition 2.7 in \cite{AD}. We show that, almost surely, excursions cut at heights do not make bubbles above any hyperplane. More precisely, we set 
\begin{equation}
    \mathscr{L} := \{u\in U^+, \; \exists 0\leq t\leq R(u), \; \exists 0\leq a<z(t), \; \Delta e_a^{(t)}(u) = 0\}. \label{eq: loop}
\end{equation}
This is the set of $u\in U^+$ making above some level an excursion which comes back to itself. Then
\begin{Prop} \label{prop: loop}
\[\n_+(\mathscr{L}) = 0.\]
\end{Prop}
\begin{proof}
We first notice that if $u\in \mathscr{L}$, then the set of $t$'s such that $\Delta e_a^{(t)}(u) = 0$ for some $0\le a<z(t)$ has positive Lebesgue measure. Therefore 
\begin{equation} \label{eq: inclusion loop}
\mathscr{L} \subset \left\{u\in U^+, \; \int_0^{R(u)} \mathds{1}_{\{\exists 0\le a<z(t), \; \Delta e_a^{(t)}(u) = 0\}} \mathrm{d}t >0\right\}.  
\end{equation}
Now using the notation in Proposition \ref{prop:Bismut}, and defining 
\[
T_a^{t,\leftarrow} := \inf\{s>0, z(t-s)=a\} \quad \text{and} \quad T_a^{t,\rightarrow} := \inf\{s>0, z(t+s)=a\},
\]
we get
\begin{align*}
    \n_+&\left(  \int_0^{R(u)} \mathds{1}_{\{\exists 0\le a<z(t), \; \Delta e_a^{(t)}(u) = 0\}} \mathrm{d}t \right) \\
    &= \overline{\n}_+\left(\{(t,u)\in \R_+\times U^+, \;  \exists 0\leq a<z(t), \; \Delta e_a^{(t)}(u) = 0\} \right) \\
    &= \overline{\n}_+\left(\{(t,u)\in \R_+\times U^+, \;  \exists 0\leq a<z(t), \; u^{t, \leftarrow}(T_a^{t,\leftarrow}) = u^{t, \rightarrow}(T_a^{t,\rightarrow}) \} \right).
\end{align*}
Bismut's description \ref{prop:Bismut} of $\n_+$ (see \cref{fig:Bismut}) finally gives
\[
\n_+\left(  \int_0^{R(u)} \mathds{1}_{\{\exists 0\le a<z(t), \; \Delta e_a^{(t)}(u) = 0\}} \mathrm{d}t \right)
=
\int_0^{+\infty} \mathrm{d}A \, \Pb\left(\exists 0<a\le A, B^{d-1}_1(T_a^1) = B^{d-1}_2(T_a^2) \right),
\]
where $B^{d-1}_1, B^{d-1}_2$ are independent $(d-1)$--dimensional Brownian motions, and $T_a^1, T_a^2$ are independent Brownian hitting times. It is now well-known that the entries of $B^{d-1}_1(T_a^1)$ and $B^{d-1}_2(T_a^2)$ are symmetric Cauchy processes in $a$. By independence, the entries of $B^{d-1}_1(T_a^1)-B^{d-1}_2(T_a^2)$ are also Cauchy processes, for which points are polar (see \cite{Ber}, Chap. II, Section 5). Hence 
\[
\n_+\left(  \int_0^{R(u)} \mathds{1}_{\{\exists 0\le a<z(t), \; \Delta e_a^{(t)}(u) = 0\}} \mathrm{d}t \right)
=
0.
\]
This yields that for $\n_+$--almost every excursion $u$, 
\[
\int_0^{R(u)} \mathds{1}_{\{\exists 0\le a<z(t), \; \Delta e_a^{(t)}(u) = 0\}} \mathrm{d}t = 0,
\]
and given the inclusion \eqref{eq: inclusion loop}, we infer that $\n_+(\mathscr{L})=0$.
\end{proof}

\medskip
\noindent \textbf{The branching property of excursions in $\Hcal_a^+$.} When cutting excursions with the hyperplanes $\Hcal_a$, the natural filtration is the one carrying the information below these hyperplanes. We call $(\Gcal_a, a\ge 0)$ this filtration, completed with the $\n_+$--negligible sets. More precisely, we first introduce the path $u^{<a}$ defined as $u^{< a}_t:=u_{\tau^{< a}_t}$ if $t< A(R(u))$ and $u^{< a}_t:= u(R(u))$ if $t=A(R(u))$, where
\begin{equation} \label{eq: A_t time change}
A(t):= \int_0^t \mathds{1}_{\{z(s) \leq a \}}\mathrm{d}s, \qquad \hbox{and} \qquad  \tau^{< a}_t:= \inf\{s> 0\,:\, A(s)>t\}.
\end{equation}
The filtration $(\Gcal_a, a\ge 0)$ is then the (completed) filtration generated by $u^{<a}$.

 Recall that we have set $T_a:= \inf\{0\le t\le R(u), \, z(t)=a\}$. Finally, we let $a>0$ and rank the excursions $(e_i^{a,+}, i\ge 1)$ in $\Hcal_a^+$ by descending order of the norm of their sizes $(\xbf^{a,+}_i, i\ge 1)$. Then the following branching property holds.

\begin{Prop} \label{prop:branching}
For all $A\in \Gcal_a$, and all nonnegative measurable functions $F_1,\ldots,F_k:U^+\rightarrow \R_+, k\ge 1,$ 
\[
\n_{+}\left(\mathds{1}_{\{T_a<\infty\}}\mathds{1}_A \prod_{i=1}^k F_i(e_i^{a,+}) \right)
=
\n_{+}\left(\mathds{1}_{\{T_a<\infty\}}\mathds{1}_A \prod_{i=1}^k \gamma_ {\xbf^{a,+}_i}[F_i]\right),
\]
and the same also holds under $\gamma_{\xbf}$ for all $\xbf\in \R^{d-1}\setminus \{0\}$.
\end{Prop}
\begin{proof}
We refer to \cite{AD} for the proof in the planar case, which is easily extended to higher dimensions. 
\end{proof}

\begin{figure} 
\bigskip
\begin{center}
\includegraphics[scale=0.8]{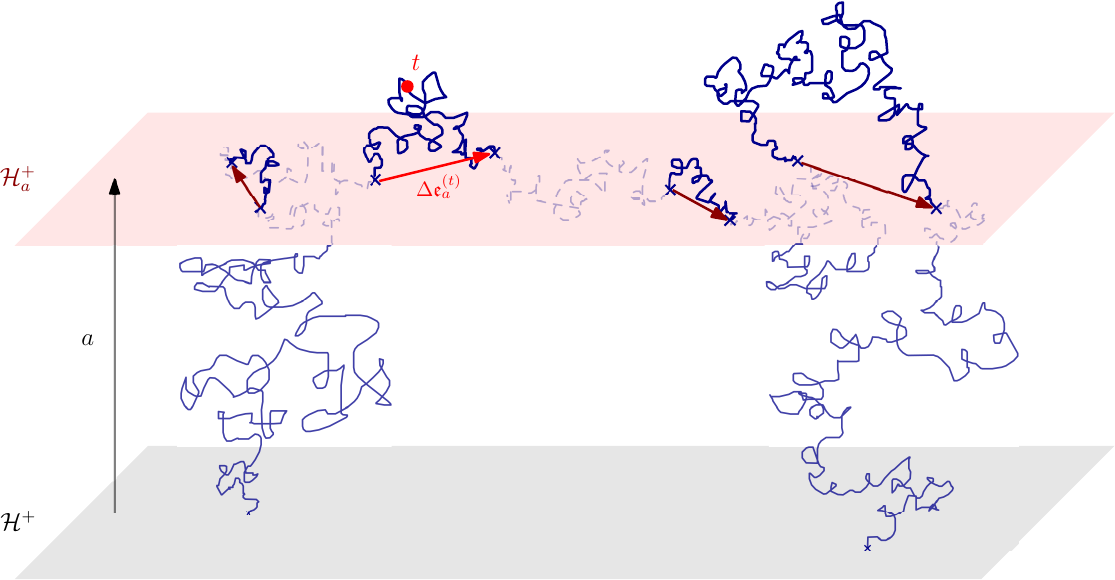}
\end{center}
\caption{Slicing at height $a$ of an excursion $u$ away from $\Hcal$. The blue trajectory represents an excursion in the half-space $\{x_d>0\}$, $d=3$. For some fixed height $a>0$ we draw the hyperplane $\Hcal_a$ and record the sub-excursions above $\Hcal_a$. The four largest of them are represented in dark blue (the reader should imagine many infinitesimal excursions). The red arrows indicate the \emph{size} of the sub-excursions, counted with respect to the orientation of $u$.} 
\label{fig:slicing excursion}
\end{figure}

%-------------------------------------------------------------------------%
%               A MANY-TO-ONE FORMULA
%-------------------------------------------------------------------------%
\subsection{A many-to-one formula}
\label{sec: many-to-one}

We now establish a key formula in the description of the process of excursions cut at heights. We first need some notation and terminology. We argue on the event that $T_a<\infty$. Let $(\ell_t^a)_{t\in [0,R(u)]}$ (resp.  $(\uptau^{a}_s)_{s\in [0,\ell^a_{R(u)}]}$) be the (resp. inverse) local time process of $u$ at level $a$ and let $(\efrak_s^a,\, s\in(0,\ell_{R(u)}^a))$ be the excursion process at level $a$ of $u$. For convenience we also define $\efrak_0^a$ and $\efrak^a_{\ell_{R(u)}^a}$ respectively as the first and last bits of excursion between $\{x_d=0\}$ and $\{x_d=a\}$. A direct consequence of Proposition \ref{prop:Markov} is that, on the event $T_a<\infty$ and conditionally on $\Fcal_{T_a}$, $(\efrak_s^a,\, s\in(0,\ell_{R(u)}^a))$ forms a Poisson point process with intensity $\n_+$ for the filtration $(\Fcal_{\uptau_s^{a}}, 0\le s\le \ell^a_{R(u)})$, stopped at the first time when an excursion hits $\{x_d=-a\}$.

Recall that $\mathscr{X}$ stands for the set of finite duration continuous trajectories in $\R^d$. We set
\begin{align*}
{\mathfrak{u}}_1^s &:= \Big( u(t), \, t\in [0, \uptau^a_{s^-}]\Big),
 \\ 
{\mathfrak{u}}_2^s &:= \Big(u(R(u)-t),\, t \in [0, R(u)-\uptau^a_s] \Big).
 \end{align*}
In words, ${\mathfrak{u}}_1^s$ and ${\mathfrak{u}}_2^s$ are the two elements of $\Xscr$ which describe the trajectory of $u$ before the excursion $\efrak^{a}_s$ and for the time-reversed trajectory of $u$ after the excursion $\efrak^{a}_s$. Finally, we use the shorthand $s^+ \in [0, \ell_{R(u)}^a]$ to denote times $0\le s \le \ell_{R(u)}^a$ such that $\efrak_s^{a} \in U^+$. 

Our description involves two processes $\mathfrak{h}_1$ and $\mathfrak{h}_2$ defined as follows. Call Bessel-Brownian excursion a process in the half-space $\Hcal^+ := \{x_d>0\}$ whose first $d-1$ entries are independent Brownian motions, and whose last coordinate $\mathfrak{z}$ is an independent three-dimensional Bessel process starting at $0$. Under $\Pb$, take $\mathfrak{h}_1$ and $\mathfrak{h}_2=\mathfrak{h}_2^\xbf$ to be independent Bessel-Brownian excursions, with $\mathfrak{h}_1(0) = 0$ and $\mathfrak{h}_2(0) = (\xbf, 0)$. We write $S_i^a := \sup\{t\ge 0\,:\, {\mathfrak z}_i(t)\le a\}$ for the last passage time at $a$ of ${\mathfrak z}_i$, $i\in\{1,2\}$.

\begin{Prop} \label{prop: key many to one}
 Let $F: \Xscr \times \Xscr \rightarrow \R_+$ be a nonnegative measurable function. Then 
\begin{equation} \label{eq: many-to-one excursions}
\gamma_{\xbf}\bigg[\mathds{1}_{\{T_a<\infty\}}\sum_{ s^+ \in [0,\ell^a_{R(u)}]} |\Delta \efrak^{a}_s|^d  F({\mathfrak u}_1^s, {\mathfrak u}_2^s) \bigg]
=
|\xbf|^d \Eb\big[F\left(({\mathfrak h}_1(t),t\in [0,S^a_1]\right),\left({\mathfrak h}_2^{\xbf}(t),t\in [0,S^a_2])\right)\big].
\end{equation}
\end{Prop}
\begin{proof}
The proof essentially follows from that of \cite[Equation (17)]{AD}. It suffices to prove \eqref{eq: many-to-one excursions} for $F(u,v)=f(u)g(v)$ with $f,g: \Xscr\rightarrow \R_+$ two measurable functions. We first deal with the left-hand side under the measure $\n_+$. By the master formula \cite[Proposition XII.1.10]{RY} for the Poisson point process $(\efrak_s^a,\, s\in(0,\ell_{R(u)}^a))$ and the disintegration property (\cref{prop:disintegration}),
\begin{align}
&\n_+\bigg( \mathds{1}_{\{T_a<\infty\}}\sum_{s^+ \in [0,\ell_{R(u)}^a]} f(\mathfrak{u}_1^s) g(\mathfrak{u}_2^s) |\Delta \efrak^{a}_s|^d
 \bigg)
 \label{beginning H excursion}\\
&= \n_+\bigg(\mathds{1}_{\{T_a<\infty\}} \int_0^{R(u)} f\left(\mathfrak{u}_1^{\ell_r^a}\right)
 \mathrm{d}\ell_r^a \nonumber \\
 & \int_{\R^{d-1}\setminus\{0\}}  \mathrm{d}\xbf \frac{\Gamma(\frac{d}{2})}{2\pi^{d/2}} \, \Eb\left[ g(\xbf+\xbf'+B^{d-1}(T^Z_{-a}-s),a+Z(T^Z_{-a}-s), 0\le s\le T^Z_{-a})\right]_{\xbf'=B^{d-1}(r)} \bigg), \nonumber %\label{H excursion}
\end{align}
where $T^Z_{-a} = \inf\{t>0, \; Z(t) = -a\}$. The change of variables $\xbf+B^{d-1}(r) \mapsto \xbf$ then shows that
\begin{multline} \label{eq: key formula change variables}
\n_+\bigg( \mathds{1}_{\{T_a<\infty\}}\sum_{s^+ \in [0,\ell_{R(u)}^a]} f(\mathfrak{u}_1^s) g(\mathfrak{u}_2^s) |\Delta \efrak^{a}_s|^d
 \bigg) 
 = 
\n_+\bigg(\mathds{1}_{\{T_a<\infty\}} \int_0^{R(u)} f\left(\mathfrak{u}_1^{\ell_r^a}\right)
 \mathrm{d}\ell_r^a \bigg) 
 \\
\cdot  \int_{\R^{d-1}\setminus\{0\}}  \frac{\Gamma(\frac{d}{2})}{2\pi^{d/2}} \mathrm{d}\xbf \Eb\left[ g(\xbf+B^{d-1}(T^Z_{-a}-s),a+Z(T^Z_{-a}-s), 0\le s\le T^Z_{-a})\right].
\end{multline}
We first argue conditionally on $Z$, where $(B^{d-1}(s),\, 0\le s\le T^Z_{-a})$ is a $(d-1)$--dimensional Brownian motion stopped at time $T^Z_{-a}$. By reversibility of Brownian motion with respect to the Lebesgue measure on $\R^{d-1}$, the ``law'' of $(\xbf+B^{d-1}_{T^Z_{-a}-s},\, 0\le s \le T^Z_{-a})$ for $\xbf$ sampled from the Lebesgue measure is that of a Brownian motion with initial measure the Lebesgue measure in $\R^{d-1}$, stopped at time $T^Z_{-a}$. Secondly, it is standard that the process $(a+Z(T^Z_{-a}-s), 0\le s\le T^Z_{-a})$ is a 3-dimensional Bessel process starting from $0$ and run until its last passage time at $a$ (see for instance \cite[Corollary VII.4.6]{RY}). The last integral in the above display therefore boils down to
\begin{multline*}
\int_{\R^{d-1}\setminus\{0\}} \frac{\Gamma(\frac{d}{2})}{2\pi^{d/2}} \mathrm{d}\xbf  \Eb\left[ g(\xbf+B^{d-1}(T^Z_{-a}-s),a+Z(T^Z_{-a}-s), 0\le s\le T^Z_{-a})\right]
\\
=
\int_{\R^{d-1}\setminus\{0\}} \frac{\Gamma(\frac{d}{2})}{2\pi^{d/2}} \mathrm{d}\xbf  \, \Eb\Big[ g({\mathfrak h}_2^{\xbf}(t),t\in [0,S_2^a])\Big].
\end{multline*}
Going back to \eqref{eq: key formula change variables}, we have
\begin{multline} \label{eq: key formula Bessel}
\n_+\bigg( \mathds{1}_{\{T_a<\infty\}}\sum_{s^+ \in [0,\ell_{R(u)}^a]} f(\mathfrak{u}_1^s) g(\mathfrak{u}_2^s) |\Delta \efrak^{a}_s|^d
 \bigg) \\
 =
\n_+\bigg(\mathds{1}_{\{T_a<\infty\}} \int_0^{R(u)} f\left(\mathfrak{u}_1^{\ell_r^a}\right)
 \mathrm{d}\ell_r^a \bigg) 
\cdot  \int_{\R^{d-1}\setminus\{0\}} \frac{\Gamma(\frac{d}{2})}{2\pi^{d/2}} \mathrm{d}\xbf  \, \Eb\Big[ g({\mathfrak h}_2^{\xbf}(t),t\in [0,S_2^a])\Big].
\end{multline}
On the other hand, by another application of the master formula \cite[Proposition XII.1.10]{RY},
\[
\n_+\bigg( \mathds{1}_{\{T_a<\infty\}} f\left(\mathfrak{u}_1^{\ell_{R(u)}^a}\right)
  \bigg)  = \n_+\bigg(\mathds{1}_{\{T_a<\infty\}} \int_0^{R(u)} f(\mathfrak{u}_1^{\ell_r^a}) \mathrm{d}\ell_r^a  \bigg)\n_-(T_{-a}<\infty). 
\]
Since $\n_-(T_{-a}<\infty) = \n_+(T_{a}<\infty)$, we conclude that 
\[
\n_+\bigg(\mathds{1}_{\{T_a<\infty\}} \int_0^{R(u)} f(\mathfrak{u}_1^{\ell_r^a}) \mathrm{d}\ell_r^a  \bigg) 
=
 \n_+\left(  f\left(\mathfrak{u}_1^{\ell_{R(u)}^a}\right) \Big| \, T_a <\infty\right).
\]
Finally, under $\n^\alpha_+(\cdot \mid T_a<\infty)$, $u$ up to its last passage time at $a$ has the law of ${\mathfrak h}_1$ up to $S_1^a$. Hence \eqref{eq: key formula Bessel} becomes
\begin{multline*}
\n_+\bigg( \mathds{1}_{\{T_a<\infty\}}\sum_{s^+ \in [0,\ell_{R(u)}^a]} f(\mathfrak{u}_1^s) g(\mathfrak{u}_2^s) |\Delta \efrak^{a}_s|^d
 \bigg) \\
 =
\Eb\Big[ f({\mathfrak h}_1(t),t\in [0,S_1^a]) \Big]
\int_{\R^{d-1}\setminus\{0\}} \frac{\Gamma(\frac{d}{2})}{2\pi^{d/2}} \mathrm{d}\xbf  \, \Eb\Big[ g({\mathfrak h}_2^{\xbf}(t),t\in [0,S_2^a])\Big].
\end{multline*}

It remains to disintegrate $\n_+$ over the endpoint using \cref{prop:disintegration}:
\begin{multline*}
\int_{\R^{d-1}\setminus\{0\}}\frac{\Gamma(\frac{d}{2})}{2\pi^{d/2}|\xbf|^d} \mathrm{d}\xbf \, \gamma_{\xbf} \bigg( \mathds{1}_{\{T_a<\infty\}}\sum_{s^+ \in [0,\ell_{R(u)}^a]} f(\mathfrak{u}_1^s) g(\mathfrak{u}_2^s) |\Delta \efrak^{a}_s|^d
 \bigg) \\
 =
\Eb\Big[ f({\mathfrak h}_1(t),t\in [0,S_1^a]) \Big]
\int_{\R^{N-1}\setminus\{0\}} \frac{\Gamma(\frac{d}{2})}{2\pi^{d/2}} \mathrm{d}\xbf  \, \Eb\Big[ g({\mathfrak h}_2^{\xbf}(t),t\in [0,S_2^a])\Big].
\end{multline*}
One can then multiply $g$ by an arbitrary function of the endpoint. This entails that for Lebesgue--almost every $\xbf\in\R^{d-1}\setminus\{0\}$,
\[
\gamma_{\xbf} \bigg( \mathds{1}_{\{T_a<\infty\}}\sum_{s^+ \in [0,\ell_{R(u)}^a]} f(\mathfrak{u}_1^s) g(\mathfrak{u}_2^s) |\Delta \efrak^{a}_s|^d
 \bigg) \\
 =
|\xbf|^d \Eb\Big[ f({\mathfrak h}_1(t),t\in [0,S_1^a]) \Big]  \cdot \Eb\Big[ g({\mathfrak h}_2^{\xbf}(t),t\in [0,S_2^a])\Big].
\]
This proves \eqref{eq: many-to-one excursions} for almost every $\xbf$. By a continuity argument that we feel free to skip, \eqref{eq: many-to-one excursions} remains true for all $\xbf\in\R^{d-1}\setminus\{0\}$.
\end{proof}

%-------------------------------------------------------------------------%
%               A CHANGE OF MEASURE
%-------------------------------------------------------------------------%
\subsection{A change of measures}
Recall from (\ref{eq:I(a) partition}) the notation $\Hcal_a^+$ for the set of excursions above $\Hcal_a$.
\begin{Thm} \label{thm:martingale}
Under $\gamma_{\xbf}$ for all $\xbf\in \R^{d-1}\setminus\{0\}$, the process
\[ 
\Mcal_a := \mathds{1}_{\{T_{a}<\infty\}} \cdot \sum_{e\in \Hcal_a^+} |\Delta e|^{d}, \quad a\ge 0,
\]
is a martingale with respect to $(\Gcal_a, a\ge 0)$.
\end{Thm}
\begin{proof}
By the branching property, it is enough to check that $\gamma_\xbf[\Mcal_a] =|\xbf|^d$. The claim then follows directly from \cref{prop: key many to one} by taking $F=1$.
\end{proof}

\medskip

We fix $\xbf\in\R^{d-1}\setminus\{0\}$.  To the martingale in Theorem \ref{thm:martingale}, we can associate the following change of measures. Recall from \eqref{eq: A_t time change} the notation $u^{<a}$. Define on the same probability space the process $({\mathfrak U}_a^\xbf,\, a> 0)$ such that for any $a> 0$, the law of ${\mathfrak U}_a^\xbf$ is that of $u^{<a}$ under the probability measure $|\xbf|^{-d} {\mathcal M}_a \mathrm{d}\gamma_\xbf$. The existence of  ${\mathfrak U}^\xbf$ results from Kolmogorov's extension theorem (the martingale property of $\Mcal$ makes this definition consistent). Our goal is to describe the law of $({\mathfrak U}_a^\xbf, a\ge 0)$. The construction involves the two processes $\mathfrak{h}_1$ and $\mathfrak{h}_2$ of \cref{sec: many-to-one}. As in \eqref{eq: A_t time change}, we introduce
\begin{equation} \label{eq: A_i(t) time change}
A_i(t):= \int_0^t \mathds{1}_{\{{\mathfrak z}_i(s) \leq a \}}\mathrm{d}s, \quad  \tau_i(t):= \inf\{s> 0\,:\, A_i(s)>t\}, \quad \text{for} \; i\in\{1,2\},
\end{equation}
and we also set 
\[
A_i(\infty) 
=
\lim_{t\to\infty} A_i(t), 
\quad i\in\{1,2\}.
\]
Under $\Pb$, we now define $\widetilde {\mathfrak U}_a^{\xbf}$ as the process obtained by concatenating $\mathfrak{h}_1$ and $\mathfrak{h}_2$ when they leave $\{x_d\le a\}$ forever, and removing everything above level $a$. More precisely, 
\[
\widetilde {\mathfrak U}_a^{\xbf}(t)
:=
\left\{
\begin{array}{ll}
{\mathfrak h}_1(\tau_1(t)) &\hbox{ if } t \in [0,A_1(\infty)), \\
{\mathfrak h}_2(\tau_2(A_1(\infty)+A_2(\infty)-t)) &\hbox{ if } t \in [A_1(\infty),A_1(\infty)+A_2(\infty)],
\end{array}
\right. 
\]
with the convention that ${\mathfrak h}_2(\tau_2(A_2(\infty))) := {\mathfrak h}_2(\tau_2(A_2(\infty))^-)$. We are now ready to describe the law of ${\mathfrak U}^\xbf$.
\begin{Thm} \label{thm:law change measures}
For any $z\ne 0$, the process $({\mathfrak U}_a^\xbf,a>0)$ is distributed as $(\widetilde {\mathfrak U}_a^\xbf,a>0)$.  
\end{Thm}
\noindent Performing the change of measure thus results in splitting the excursion into two independent excursions in the half-space $\Hcal^+$ going to infinity, as in \cref{fig:change measures}.

\medskip
\begin{proof}
This follows readily from \cref{prop: key many to one} by taking $F(\mathfrak{u}_1^s, \mathfrak{u}_2^s)$ as a measurable function of $u^{<a}$.
\end{proof}

\subsection{Proof of Theorem \ref{thm:exc}}
We reformulate the previous results in the parlance of Section \ref{sec: spatial GF}. Setting 
\[
\mathbf{Z}(a) := \left\{\left\{ \Delta e, \; e\in\Hcal_a^+ \right\}\right\}, \quad a\ge 0,
\]
it follows from \cref{prop:branching} that $\mathbf{Z}$ enjoys a branching property akin to \cref{prop: temporal branching spatial GF}. We could have pointed out an Eve cell in the spirit of \cite[Theorem 3.3]{AD} by considering the locally largest excursion. Together with an avatar of \cite[Theorem 3.6]{AD}, this proves that under $\gamma_{\xbf}$, $\mathbf{Z}$ is a \emph{spatial growth-fragmentation process}. Actually, one should first check that the evolution of the Eve cell generates all the excursions, but this is a simple consequence of the arguments presented in \cite[Theorem 4.1]{AD}. In the previous exposition, we chose to rather dwell on the spine description. More specifically, the martingale in \cref{thm:martingale} is a temporal version of the martingale in \cref{thm:M(n)}. Then, \cref{thm:law change measures} determines the law of the spine without reference to \cref{thm:spine}. The spine is described as the Brownian motion $B^{d-1}$ taken at the hitting times of another independent linear Brownian motion, and hence is a $(d-1)$--dimensional isotropic Cauchy process.

\begin{figure} 
\begin{center}
\includegraphics[scale=0.7]{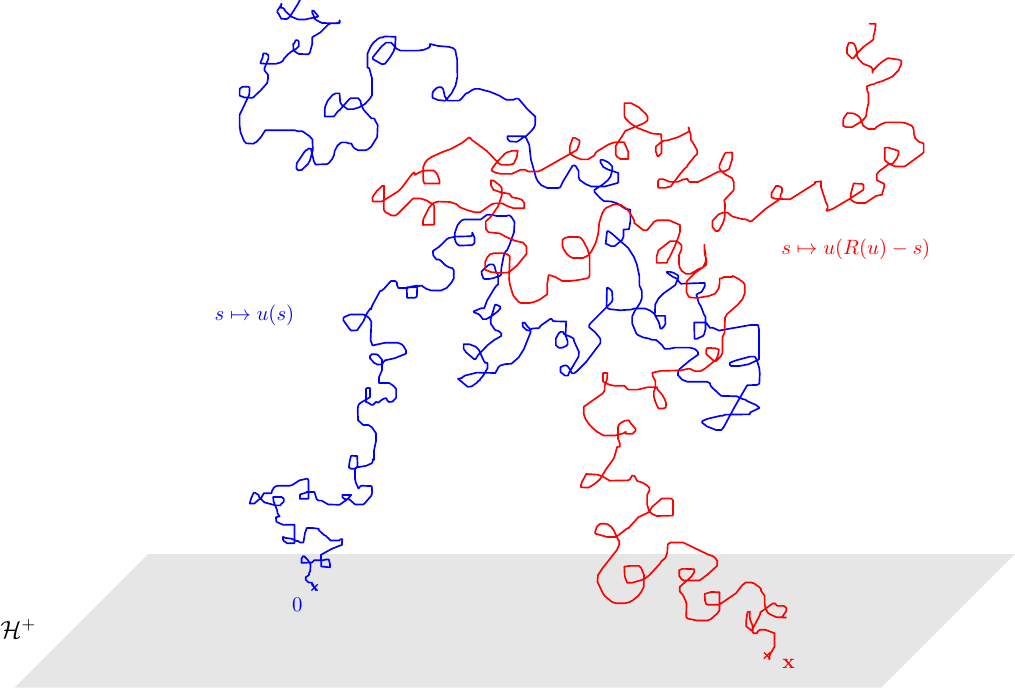}
\end{center}
\caption{Splitting of the excursion under the change of measure. Under the change of measure, the excursion splits into two independent Bessel-Brownian excursions (blue and red). For fixed height $a>0$, the sub-excursion above $a$ straddling the point at infinity is obtained by running two independent $d$--dimensional Brownian motions started from infinity and stopped when hitting the hyperplane $\Hcal_a$. }
\label{fig:change measures}
\end{figure}

%-------------------------------------------------------------------------%
%               EXTENSION TO STABLE PROCESSES
%-------------------------------------------------------------------------%
\subsection{Extension to isotropic stable Lévy processes}
As in \cite{DS}, we can extend the previous construction to stable processes with index $\alpha\in(0,2)$. We recall that we have set $d\ge 3$, and that the  case $d=2$ was already treated in \cite{DS}. We will not provide all the details of the proofs since the arguments are similar to the Brownian case described above.

\medskip
\noindent \textbf{The excursion measure $\nalpha$.} We shall consider the following excursions, which consist in replacing the first $(d-1)$ entries of the previous setting by an isotropic $\alpha$--stable Lévy process in $\R^{d-1}$. We keep the notation in \cref{sec:excursion measure}, except that now is defined on the probability space a $(d-1)$--dimensional isotropic stable Lévy process $X^{d-1}$, and we consider the process $Z^d :=(X^{d-1},Z)$ with Brownian last coordinate. Then, we introduce the excursion process $(\ealpha_s,s>0)$ as
\begin{itemize}
    \item[(i)] if $\uptau_s-\uptau_{s^-}>0$, then 
    \[\ealpha_s : r\mapsto \left(X^{d-1}_{r+\uptau_{s^-}}-X^{d-1}_{\uptau_{s^-}}, Z_{r+\uptau_{s^-}}\right), \quad r\leq \uptau_s-\uptau_{s^-},\]
    \item[(ii)] if $\uptau_s-\uptau_{s^-}=0$, then $\ealpha_s = \partial$.
\end{itemize}
As in \cref{prop:excursion measure}, this defines a Poisson point process with intensity measure 
\[
\nalpha(\mathrm{d}u',\mathrm{d}z) := n(\mathrm{d}z) \Pb\Big((X^{d-1})^{R(z)}\in \mathrm{d}u'\Big).
\]
Let $\nalpha_+$ be the restriction of $\nalpha$ to positive excursions. We now want to condition $\nalpha_+$ on the endpoint of the excursion. For $\xbf\in\R^{d-1}$ and $r>0$, let $\Pbb{r}{\alpha, 0}{\xbf}$ denote the law of an $\alpha$--stable bridge from $0$ to $\xbf$ over $[0,r]$. In addition, we write $(p^{\alpha}_r, r\ge 0)$ for the transition densities of $X^{d-1}$. Throughout this section, we fix $\omega_d:= d-1+\frac{\alpha}{2}$.

\begin{Prop} \label{prop: disintegration stable}
The following disintegration formula holds:
\begin{equation} \label{eq: disintegration stable}
\nalpha_+ = \int_{\R^{d-1}\setminus \{0\}} \mathrm{d}\xbf \, \frac{C_d}{|\xbf|^{\omega_d}} \cdot \gammaalpha_{\xbf},
\end{equation}
where $\gammaalpha_{\xbf}$, $\xbf\in\R^{d-1}\setminus\{0\}$, are probability measures, and
\[
C_d = \frac{\alpha}{2\sqrt{2\pi}} \int_{\R_+} \mathrm{d}v \, p_1^{\alpha}(v\cdot\mathbf{1}) v^{\omega_d-1},
\]
where $\mathbf{1}$ denotes the ``north pole'' in $\Sb^{d-2}$. 
In addition, for all $\xbf\in\R^{d-1}\setminus\{0\}$,
\[
\gammaalpha_{\xbf} = \int_0^{\infty} \mathrm{d}r \frac{p_1^{\alpha}(r^{-1/\alpha}v\cdot\mathbf{1})}{2\sqrt{2\pi}r^{1+\frac{\omega_d}{\alpha}}} \Pbb{r|\xbf|^2}{\alpha, 0}{\xbf} \otimes \Pi_{r|\xbf|^2}.
\]
\end{Prop}
\begin{proof}
Let $f$ and $g$ be two nonnegative measurable functions, respectively defined on $\Xscr^{d-1}$ and $\Xscr_0$. Following the proof of \cref{prop:disintegration}, we end up with
\begin{align*}
\int_{U^+} f(u')g(z)\nalpha_+(\mathrm{d}u',\mathrm{d}z)
&=
\int_0^{\infty} \frac{\mathrm{d}r}{2\sqrt{2\pi r^3}} \Pi_r[g]\Eb[f((X^{d-1})^r)] \\
&= \int_0^{\infty} \frac{\mathrm{d}r}{2\sqrt{2\pi r^3}} \int_{\R^{d-1}} \mathrm{d}\xbf \, p_r^{\alpha}(\xbf) \Pi_r[g]\Ebb{r}{\alpha, 0}{\xbf}[f(X^{d-1})].
\end{align*}
Note that, by self-similarity, for all $r>0$ and $\xbf\in\R^{d-1}$,
\[
p_r^{\alpha}(\xbf) = r^{-\frac{d-1}{\alpha}} p_1^{\alpha}(r^{-1/\alpha}\cdot\xbf).
\]
Hence
\[
\int_{U^+} f(u')g(z)\nalpha_+(\mathrm{d}u',\mathrm{d}z)
=
\int_0^{\infty} \frac{\mathrm{d}r}{2\sqrt{2\pi r^3}} \int_{\R^{d-1}} \mathrm{d}\xbf \, r^{-\frac{d-1}{\alpha}} p_1^{\alpha}(r^{-1/\alpha}\cdot\xbf) \Pi_r[g]\Ebb{r}{\alpha, 0}{\xbf}[f(X^{d-1})],
\]
and by the change of variables $u(r):=\frac{r}{|\xbf|^{\alpha}}$, this is
\begin{multline*}
\int_{U^+} f(u')g(z)\nalpha_+(\mathrm{d}u',\mathrm{d}z) \\
=
\int_{\R^{d-1}} \frac{\mathrm{d}\xbf}{|\xbf|^{\omega_d}} \int_0^{\infty} \frac{\mathrm{d}u}{2\sqrt{2\pi u^3}} u^{-\frac{d-1}{\alpha}} p_1^{\alpha}\Big(u^{-1/\alpha}\cdot\frac{\xbf}{|\xbf|}\Big) \Pi_{u|\xbf|^{\alpha}}[g]\Ebb{u|\xbf|^{\alpha}}{\alpha, 0}{\xbf}[f(X^{d-1})].
\end{multline*}
Observe that the isotropy of $X^{d-1}$ yields the relationship $p_1^{\alpha}\Big(u^{-1/\alpha}\cdot\frac{\xbf}{|\xbf|}\Big) = p_1^{\alpha}(u^{-1/\alpha}\cdot \mathbf{1})$, so that
\[
\int_{U^+} f(u')g(z)\nalpha_+(\mathrm{d}u',\mathrm{d}z) \\
=
\int_{\R^{d-1}} \frac{\mathrm{d}\xbf}{|\xbf|^{\omega_d}} \int_0^{\infty} \mathrm{d}u \frac{p_1^{\alpha}(u^{-1/\alpha}\cdot \mathbf{1})}{2\sqrt{2\pi}u^{1+\frac{\omega_d}{\alpha}}}  \Pi_{u|\xbf|^{\alpha}}[g]\Ebb{u|\xbf|^{\alpha}}{\alpha, 0}{\xbf}[f(X^{d-1})].
\]
The proposition follows.
\end{proof}
\begin{Rk}
We emphasize that the proof of \cref{prop: disintegration stable} uses the isotropy assumption on $X^{d-1}$, and indeed formula \eqref{eq: disintegration stable} shows that the excursion measure $\nalpha_+$ assigns a weight to the endpoint $\xbf$ which only depends on its radial part $|\xbf|$. If $X^{d-1}$ were not isotropic, then one would have to deal with the angular part of $\xbf$ in the disintegration.
\end{Rk}
The following proposition is a Bismut description of $\nalpha_+$, which is easily extended from \cref{prop:Bismut}. The picture looks roughly the same as in \cref{fig:Bismut}, albeit the two trajectories have their first $(d-1)$ entries distributed as an isotropic stable process in $\R^{d-1}$.
\begin{Prop} \label{prop:Bismut stable} (Bismut's description of $\nalpha_+$)

\noindent Let $\overline{\nalpha}_+$ be the measure defined on $\R_+\times U^+$ by
\[\overline{\nalpha_+}(\mathrm{d}t,\mathrm{d}u) = \mathds{1}_{\{0\leq t\leq R(u)\}} \mathrm{d}t \, \nalpha_+(\mathrm{d}u).\]
Then under $\overline{\nalpha_+}$ the "law" of $(t,(u',z))\mapsto z(t)$ is the Lebesgue measure $\mathrm{d}A$ on $\R_+$, and conditionally on $z(t)=A$, $u^{t, \leftarrow}=\left(u(t-s)-u(t)\right)_{0\leq s\leq t}$ and $u^{t,\rightarrow}=\left(u(t+s)-u(t)\right)_{0\leq s\leq R(u)-t}$ are independent and evolve as $Z^d$ killed when reaching the hyperplane $\left\{x_d=-A\right\}$. 
\end{Prop}
\noindent One of the consequences of this decomposition is that for $\nalpha_+$--almost every excursion, there is no loop above any level. More precisely, recall the definition of $\mathscr{L}$ in \eqref{eq: loop}. Then $\nalpha_+(\mathscr{L}) =0$. The proof can be taken \textit{verbatim} from \cref{prop: loop}, using that a stable process in dimension $d-1\ge 2$ does not hit points (see \cite[II, Corollary 17]{Ber}).
\bigskip

\begin{figure} 
\begin{center}
\includegraphics[scale=0.8]{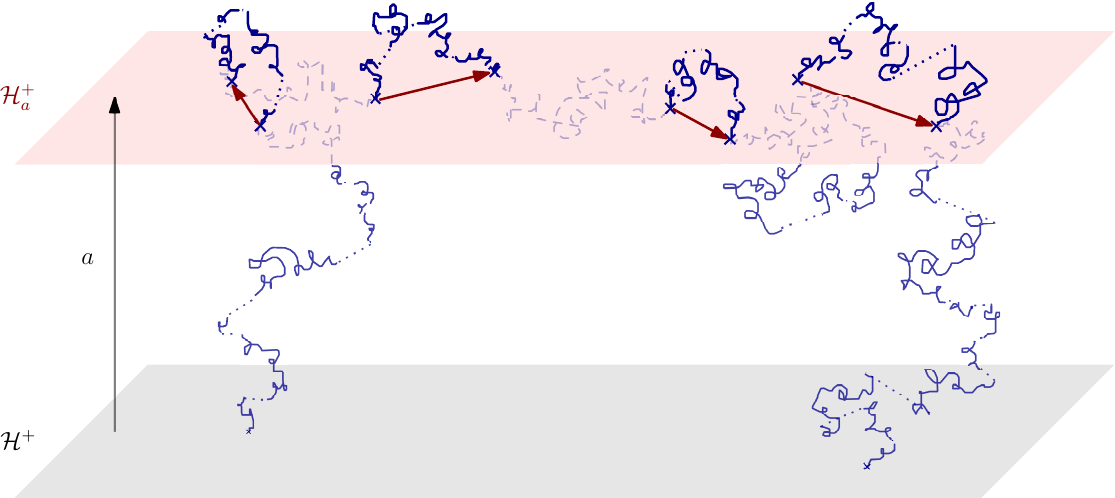}
\end{center}
\caption{Slicing of an excursion in $\Hcal_+$ with stable first two coordinates, in dimension $d=3$. The excursion is drawn in blue. The trajectory is càdlàg but jumps never occur for the height. We record the \emph{length} (in red) of the sub-excursions (in dark blue) made above $\Hcal_a$.}
\label{fig:slicing excursion stable}
\end{figure}

\bigskip
\noindent \textbf{The branching property under $\nalpha_+$.} We will be interested in cutting excursions with hyperplanes at varying heights, and study the \emph{length} of the subexcursions above these hyperplanes (\cref{fig:slicing excursion stable}). As in \cref{prop:branching}, this exhibits a branching structure that we summarise in the next result, in the language introduced in \cref{sec: slicing excursions}.
\begin{Prop}
For all $A\in \Gcal_a$, and all nonnegative measurable functions $F_1,\ldots,F_k:U^+\rightarrow \R_+, k\ge 1,$ 
\[
\nalpha_{+}\left(\mathds{1}_{\{T_a<\infty\}}\mathds{1}_A \prod_{i=1}^k F_i(e_i^{a,+}) \right)
=
\nalpha_{+}\left(\mathds{1}_{\{T_a<\infty\}}\mathds{1}_A \prod_{i=1}^k \gammaalpha_ {\xbf^{a,+}_i}[F_i]\right),
\]
and the same also holds under $\gammaalpha_{\xbf}$ for all $\xbf\in \R^{d-1}\setminus \{0\}$.
\end{Prop}

\medskip
\noindent \textbf{Martingale and spine decomposition under $\gammaalpha_{\xbf}$.} 
In line with \cref{thm:martingale} and \cref{thm:law change measures}, we reveal the martingale in the stable setting and describe the law after the change of measure. The notation is implicitly taken from the Brownian case. All the proofs are omitted because they are simple extensions of their Brownian analogues, going through a many-to-one formula akin to \cref{prop: key many to one}. Recall that $\omega_d = d-1+\frac{\alpha}{2}$.
\begin{Thm}
Under $\gammaalpha_{\xbf}$ for all $\xbf\in \R^{d-1}\setminus\{0\}$, the process
\[ 
\Mcal^{\alpha}_a := \mathds{1}_{\{T_{a}<\infty\}} \cdot \sum_{e\in \Hcal_a^+} |\Delta e|^{\omega_d}, \quad a\ge 0,
\]
is a martingale with respect to $(\Gcal_a, a\ge 0)$.
\end{Thm}

\bigskip
\begin{figure} 
\begin{center}
\includegraphics[scale=0.8]{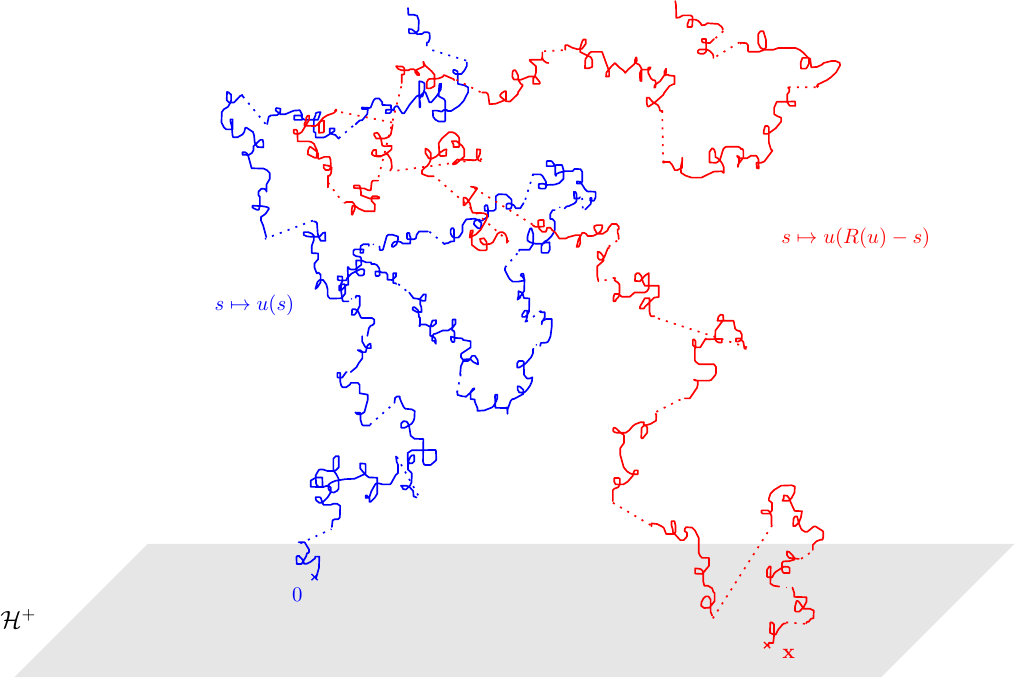}
\end{center}
\caption{The excursion $u$ seen under $\mu_{\xbf}^{\alpha}$. Under the change of measure, $u$ splits into two independent $(\alpha,\Hcal_+)$--excursions (blue and red), which are the analogues of the Brownian half-space excursions appearing in \cref{fig:change measures}, when the first $(d-1)$ coordinates are replaced with an isotropic stable process. The length of the sub-excursion above some height $a$ straddling the point at infinity is obtained by subordinating the isotropic process at the Brownian hitting time of level $a$. Let us stress once more that the last coordinate is continuous, so that this length is well defined for all positive height $a>0$.}
\label{fig:change measures stable}
\end{figure}

\noindent Let $\xbf\in\R^{d-1}\setminus\{0\}$. Consider the change of measure \added{$\mu^{\alpha}_{\xbf}$} such that
\[
\frac{\mathrm{d}\mu^{\alpha}_{\xbf}}{\mathrm{d}\gammaalpha_{\xbf}}\bigg|_{\Gcal_a} := \frac{\Mcal^{\alpha}_a}{|\xbf|^{\omega_N}}, \quad a\ge 0.
\]
We now come to the description of the excursion under $\mu^{\alpha}_{\xbf}$. Call $(\alpha,\Hcal^+)$--excursion a process in $\R^d$ whose first $(d-1)$ entries form an isotropic $\alpha$--stable Lévy process, and whose last entry is an independent $3$--dimensional Bessel process starting at $0$ (so that this process actually remains in $\Hcal^+$). We set $T_a:= \inf\{0\le t\le R(u), \, z(t)=a\}$, and $S_a:= \inf\{0\le t\le R(u), \, z(R(u)-t)=a\}$. 

\begin{Thm} \label{thm:law change measures stable}
Under $\mu^{\alpha}_{\xbf}$, for all $a>0$, the processes $(u(s), s\le T_a)$ and $(u(R(u)-s), s\le S_a)$ are independent $(\alpha,\Hcal^+)$--excursions started respectively from $0$ and $(\xbf,0)$ and stopped when hitting $\Hcal_a$.
\end{Thm}

\cref{fig:change measures stable} illustrates the theorem.

\medskip
The proof of Theorem \ref{thm:excal} then follows similarly as in the proof of Theorem \ref{thm:exc}. Indeed, we note that the process
\[
\mathbf{Z}(a) := \left\{\left\{ \Delta e, \; e\in\Hcal_a^+ \right\}\right\}, \quad a\ge 0,
\]
is a spatial growth-fragmentation process under $\gammaalpha_{\xbf}$. One could fiddle with the ideas of \cite[Theorem 6.8]{DS} in order to define an Eve cell process driving $\mathbf{Z}$, but beware that the (signed) growth-fragmentation process described therein is not isotropic as such (one needs to adjusts the constants $c_+$ and $c_-$ to recover an isotropic process). \cref{thm:law change measures stable} provides the law of the spine as an isotropic $(d-1)$--dimensional $\frac{\alpha}{2}$--stable process.

%-------------------------------------------------------------------%
%               BIBLIOGRAPHY
%-------------------------------------------------------------------%

\bibliography{biblio}
\bibliographystyle{alpha}

\end{document}